\documentclass[Threeside,11pt]{article}
\usepackage{etoolbox}
\usepackage{adjustbox}
\makeatletter
\patchcmd{\thebibliography}
{\list}
{\small          
	\setstretch{1.2}
	\list}
{}{}
\makeatother
\usepackage{setspace} 
\usepackage{tikz}
\usepackage [latin1]{inputenc}
\usepackage[english]{babel}
\usepackage{amsmath}
\usepackage{amsthm}
\usepackage{algorithm}
\usepackage{cite}
\usepackage{amsfonts}
\usepackage{amssymb}
\usepackage{mathrsfs}
\usepackage{graphicx}
\usepackage{colortbl,dcolumn}
\usepackage{marvosym}
\usepackage{ifsym}
\usepackage{paralist}
\usepackage{xcolor}
\usepackage{array,color}
\usepackage[
colorlinks,
citecolor=blue,
linkcolor=blue,
backref=page
allcolors=black
]{hyperref} 
\graphicspath{{figures/}}
\allowdisplaybreaks

       \newtheorem{lemma}{\bf Lemma}[section]
       \newtheorem{theorem}{\bf Theorem}[section]

       \newtheorem{definition}{\bf Definition}[section]
       \newtheorem{remark}{\bf Remark}[section]

       \numberwithin{equation}{section}

\begin{document}
\title{{\sl Weighted multiple ergodic averages via analytic observables over $ \mathbb{T}^\infty $: Is exponential pointwise convergence universal?
}}
\author{Zhicheng Tong $^{\mathcal{z}}$, Yong Li $^{\mathcal{x}}$}

\renewcommand{\thefootnote}{}
\footnotetext{\hspace*{-6mm}
\begin{tabular}{l l}
$^\mathcal{z}$~School of Mathematics, Jilin University, Changchun 130012, P. R.  China. \url{tongzc25@jlu.edu.cn}\\	$^{\mathcal{x}}$~The corresponding author. School of Mathematics, Jilin University, Changchun 130012, P. R.  China; \\Center for Mathematics and Interdisciplinary Sciences,  Northeast  Normal  University, Changchun 130024, \\P. R. China. \url{liyong@jlu.edu.cn}
\end{tabular}}

\date{}
\maketitle

\begin{abstract}
By employing an accelerated weighting method, we establish arbitrary polynomial and exponential pointwise convergence for multiple ergodic averages under general balancing conditions in both discrete and continuous settings, including quasi-periodic and almost periodic cases. We also present joint Diophantine rotations as explicit applications. Specifically, for the first time, by excluding nearly rational rotations with zero measure, we address the fundamental question of whether exponential pointwise convergence via analytic observables is universal, even when multiplicatively averaging over the infinite-dimensional torus $ \mathbb{T}^\infty $. We achieve this by introducing an innovative approach that effectively overcomes the previous difficulties. Moreover, by constructing counterexamples concerning multiple ergodicity, we highlight the indispensability of the joint nonresonance and establish the optimality of our weighting method in preserving rapid convergence. We also provide numerical simulations and analysis to further illustrate and validate our results.
\end{abstract}

\tableofcontents

\section{Introduction}
\renewcommand{\thefootnote}{\fnsymbol{footnote}}
This paper mainly concerns the acceleration of weighted Birkhoff averages driven by \textit{almost all} rotations in the \textit{multiple} sense, from classical (unweighted type) one-order polynomial  convergence ($ \mathcal{O}(N^{-1}) $) to \textit{exponential convergence}. As another main  novelty, we demonstrate, for the first time, that  exponential pointwise convergence is \textit{universal} via analytic observables. Below, we shall review the history of ergodic theory and the origins of the acceleration method\footnote{This  can also be termed the ``\textit{weighted quasi-Monte Carlo method}''; see \cite{TL25c} for details.}, and elucidate the fundamental contributions of our paper in detail.

Arising from statistical mechanics and celestial mechanics, the classical ergodic theorem established by Birkhoff \cite{Birkhoff} and von Neumann \cite{vonNeumann} states  that the time average of a function $ f $ evaluated along a trajectory of length $ N $ converges to the spatial average via ergodicity, which is known as one of the most fundamental and important problems in the theory of dynamical systems. For further insights, one can refer to readable survey articles by Mackey \cite{MR0346131} and Moore \cite{MR3324732}. To be more precise, consider a map $ T: X \to X $ on a topological space $ X $ with a  probability measure $ \mu $ for which $ T $ is invariant. Then, for a fixed initial point $ x \in X $ and an observable $ f $ on $ X $, the long time average of $ f $ is expressed as
\[{\mathrm{B}_N}\left( f \right)\left( x \right): = \frac{1}{N}\sum\limits_{n = 0}^{N - 1} {f\left( {{T^n}\left( x \right)} \right)},\]
 which we call the Birkhoff average of $ f $. The classical ergodic theorem indicates  that $ {\mathrm{B}_N}\left( f \right)\left( x \right) $ will converge to the spatial average $ \int_X {fd\mu }  $ in a suitable way (in the $ L^2 $ norm or a.e.), assuming  ergodicity for $ T $ and certain regularity conditions for the observable $ f $ (such as $ L^2 $ or $ L^1 $ integrability).  Further, much effort has been made in investigating  \textit{multiple ergodic averages} since Furstenberg's renowned work \cite{MR0498471,MR0603625}; see, for instance,   Bourgain \cite{MR1037434},   Tao \cite{MR2408398},   Fan et al. \cite{MR3488037}, Fan \cite{MR3693506,Fanarxiv}, Fan et al. \cite{MR3848423},   Huang et al. \cite{MR4041103,MR3996054}, and the references therein.    One of the primary motivations for exploring the results in this paper stems from these fundamental classical multiple ergodic theorems.

It is well known since Krengel's  original observation  \cite{MR0510630} that the convergence rate of Birkhoff averages in ergodic theory can be very slow in general settings, \textit{even being  arbitrarily slow for certain counterexamples.} Recently, Ryzhikov \cite{MR4602432} established similar yet distinct statements through an elegant approach. Very recently, he also investigated the weighted version (such as the weights in this paper) in \cite{Ryzhikovarxiv,Ryzhikovarxiv2,Ryzhikovarxiv3}\footnote{In other words, these intriguing new works indicate that the weighting approach employed in this paper does not significantly accelerate the convergence of Birkhoff averages for general dynamical systems, which is intuitive. In our subsequent research, we will demonstrate that for certain dynamical systems (not limited to quasi-periodic and almost periodic cases), the convergence of Birkhoff averages can be accelerated polynomially or exponentially.}. We also mention the counterexamples constructed by 	Yoccoz \cite{MR604672,MR1367354} based on extremely Liouvillean (nearly rational) rotations over the finite-dimensional torus. Such slow convergence is indeed \textit{universal} and \textit{cannot be avoided} in ergodic theory, and it would be at most  $\mathcal{O}(N^{-1})$ in non-trivial cases, i.e., the observables are  non-constant, see Kachurovski\u{\i} \cite{MR1422228} for details. We would also like to mention other fundamental results on Birkhoff averages and related ergodic problems with slow convergence rates. For example, see del Junco and Rosenblatt \cite{MR553340}, Fan \cite{Fan19}, Kakutani and Petersen \cite{MR612450}, Przytycki et al. \cite{MR1005606}, Kachurovski\u{\i} and Podvigin \cite{MR3643963,MR3981324},  Podvigin \cite{MR4440287,MR4904729,MR4767798,Podviginnew},  Ryzhikov \cite{MR4843319}, and the references cited therein. And even more frustrating, aiming to achieve high-precision numerical results, some computations may require time spans of billions of years, as discussed by Das and Yorke in  \cite[Section 1.9]{MR3755876}.

 Weighting methods, in light of the acknowledged slow convergence, are therefore extremely important in accelerating computations in both mathematics and mechanics. There has been active current interest in finding appropriate weighting functions to improve the convergence rate of the corresponding ergodic averages. To investigate quasi-periodic perturbations of quasi-periodic flows in \cite{MR1234445,MR1222935,MR1720890} and others, Laskar utilized  a weighting function $ \sin^2(\pi x) $ to accelerate the rate of computations. Additionally, he claimed that a particular exponential weighting function had excellent asymptotic properties without implementing it or demonstrating its convergence properties, see \cite{MR1720890} (Remark 2 in Appendix, p.146). Notably,  the resulting convergence rate could be proved to be  faster than an arbitrary polynomial  type, as we shall detail later. To be more precise, he utilized  the following weighting function to study the ergodicity in dynamical systems:
\begin{equation}\label{expfun}
	w\left( x \right) = \exp \left( { - {x^{ - 1}}{{\left( {1 - x} \right)}^{ - 1}}} \right) \cdot {\left( {\int_0^1 {\exp \left( { - {t^{ - 1}}{{\left( {1 - t} \right)}^{ - 1}}} \right)dt} } \right)^{ - 1}}
\end{equation}
on $ (0,1) $, and $ w(x)=0 $ on $ \mathbb{R}\backslash (0,1) $. It is evident  that $ w \in C_0^\infty \left( {\left[ {0,1} \right]} \right) $ and $ \int_0^1 {w\left( x \right)dx} =1 $. Viewed from a statistical perspective, this approach effectively reduces the influence of the initial and final data, thereby accentuating the data in the middle. This emphasis aligns with the concept of averaging, consequently yielding intuitive rapid convergence. 

 In the recent work \cite{MR4768308}, the authors  considered the weighted Birkhoff average below (denoted as $ \mathrm{WB}_N $ for brevity, including  the continuous case), 
\begin{equation}\label{WB}
	\mathrm{WB}_N\left( f \right)\left( \theta  \right): = \frac{1}{{{A_N}}}\sum\limits_{s = 0}^{N - 1} {w\left( {s/N} \right)f\left( {{\mathcal{T}}_\rho ^n\left( \theta  \right)} \right)},
\end{equation}
where $ {A_N} = \sum\nolimits_{s = 0}^{N - 1} {w\left( {s/N} \right)}  $ with $ N \in \mathbb{N}^+ $  sufficiently large, $ \mathcal{T}_\rho: \theta \to \theta +\rho \mod 1 $ in each coordinate is a rotation map with $ \rho $ defined on the finite-dimensional  torus $ \mathbb{T}^d:=[0,1]^d $ (or the infinite-dimensional torus $ \mathbb{T}^{\infty}:=[1,2]^{\mathbb{N}} $, here we denote $ d= \infty $ for brevity), and the observable $ f $ belongs to some Banach space $ \mathscr{B} $ with certain regularity. Under specific  restrictions, the authors achieved  arbitrary polynomial and  even exponential convergence for $ \mathrm{WB}_N $. This advances the findings of Das and Yorke  \cite{MR3755876}, who initially provided a rigorous analysis of  arbitrary polynomial convergence for the  quasi-periodic case with Diophantine rotations on $\mathbb{T}^d$ using this weighting method.  The authors also investigated the Ces\`aro weighted Birkhoff averages in \cite{CESAROTL} and the weighted averages along decaying waves in \cite{arXiv:2408.09398}.  However, to the best of our knowledge, there have been \textit{very few} works concerning almost periodicity even in  classical ergodic theory, let along accelerating the convergence rate. Given that rotations have an infinite number of components, the torus $ \mathbb{T}^\infty $ and Fourier analysis on it must admit specific  spatial structures. Consequently, almost periodic problems pose significantly greater challenges compared to quasi-periodic ones. In \cite{MR4768308}, the  authors established the universality of exponential convergence in the  quasi-periodic case, through  Diophantine rotations and analyticity. However, in the almost periodic context, a \textit{much stronger} regularity condition than analyticity is required for observables to achieve exponential convergence for infinite-dimensional Diophantine nonresonance. For example, observables need to exhibit super-exponential decay for their Fourier coefficients. This demanding requirement poses significant limitations for practical applications. These limitations arise mainly from the challenges associated with the infinite-dimensional spatial structure and the presence of small divisors.
While it is possible to demonstrate that arbitrary polynomial convergence is universal in the almost periodic context, this type of convergence fundamentally differs from true exponential convergence. Specifically, the control coefficient tends to infinity with increasing order of the polynomial in convergence. Given these considerations and drawing on inspiration from the historical multiple ergodic theorems, it is imperative for this paper to address the following fundamental questions:

\begin{itemize}
	\item[\textbf{(Q1)}] \textbf{Under the accelerated weighting method, do the multiple ergodic averages still exhibit rapid convergence (arbitrary polynomial or even exponential convergence)?}
	
	\item[\textbf{(Q2)}] \textbf{When the observables are analytic, can we establish the universality of exponential convergence (in a full measure sense) for both quasi-periodic and almost periodic cases?}
\end{itemize}

These questions are challenging, especially the second question (Q2). It is important to note that existing strategies  are \textit{not applicable} to address (Q2).
Let us consider the  weighted multiple Birkhoff average with rotations on the torus in the discrete case, denoted by $ \mathrm{DMW}_N^\ell $ for short since it can be shown   that the limit is independent of the initial point $ \theta $ via certain regularity of $ \{F_j\}_{j=1}^{\ell} $ (actually such $ \theta $ in the summation could be different, but we do not pursue that):
\begin{equation}\label{DMW}
	\mathrm{DMW}_N^\ell\left( \mathcal{F} \right)\left( \theta  \right): = \frac{1}{{{A_N}}}\sum\limits_{n = 0}^{N - 1} {w\left( {n/N} \right){F_1}\left( {\mathscr{T}_{{\rho _1}}^n\left( \theta  \right)} \right) \cdots {F_\ell }\left( {\mathscr{T}_{{\rho _\ell }}^n\left( \theta  \right)} \right)}
\end{equation}
with $ {A_N} = \sum\nolimits_{s = 0}^{N - 1} {w\left( {s/N} \right)}  $ and sufficiently large $ N \in \mathbb{N}^+ $, where $ \mathscr{T}_{\rho_{j}}^s: \theta \to \theta +s\rho_j \mod 1$  in each coordinate with $ s>0 $ is a rotation map with $ \rho_j $ defined on the finite-dimensional torus $ \mathbb{T}^d:=[0,1]^d $ (or the infinite-dimensional torus $ \mathbb{T}^{\infty}$), and $ F_j $ belongs to some Banach space $ \mathscr{B} $ with the algebra property (the definition of the product of functions will be given later), where $ 1 \leqslant j \leqslant \ell \in \mathbb{N}^+$. Here, we  denote by $\tilde{\rho}:=\left(\rho_1,\ldots,\rho_\ell\right)\in \mathbb{T}^{d\ell}$  the \textit{joint rotation} of the given rotations $ \{\rho_j\}_{j=1}^\ell $ for convenience. As for the continuous case, one could similarly define the following weighted integral in the multiple sense  with $ T>0 $ sufficiently large, abbreviated as $ \mathrm{CMW}_T^\ell $: 
\begin{equation}\label{CMW}
	\mathrm{CMW}_T^\ell\left( \mathcal{F} \right)\left( \theta  \right): = \frac{1}{T}\int_0^T {w\left( {s/T} \right){F_1}\left( {\mathscr{T}_{{\rho _1}}^s\left( \theta  \right)} \right) \cdots {F_\ell }\left( {\mathscr{T}_{{\rho _\ell }}^s\left( \theta  \right)} \right)ds}.
\end{equation}

The \textbf{main contributions} of this paper are as summarized  follows. Under various \textit{balancing conditions} on the joint rotation $ \tilde{\rho} $ and Fourier coefficients of all $ F_j $ with $ 1 \leqslant j \leqslant \ell $, we establish  that both $ \mathrm{DMW}_N^\ell $ and $ \mathrm{CMW}_T^\ell $ converge pointwise to the product of spatial averages
\[\prod\limits_{j = 1}^\ell  \left({\int_{{\mathbb{T}^d}} {{F_j}(\hat \theta )d\hat \theta } }\right),\;\; 1 \leqslant d \leqslant \infty \]
at an (arbitrary) polynomial or exponential rate. In particular, we develop strategies and weaken the original restrictions in \cite{MR4768308} to demonstrate, for the first time, \textit{the universality of exponential pointwise  convergence for multiple ergodic averages in the   almost  periodic context}, by constructing a nonresonant condition with full probability measure or enhancing  Diophantine estimates through   truncation. The concept of universality means that the conclusion holds for almost all rotations, which is \textit{not feasible} with the original methods. This fully addresses questions (Q1) and (Q2), and also explains Laskar's simulation findings regarding the quasi-periodic case, specifically  the convergence rate faster than an arbitrary polynomial type. 
It is worth reiterating that the present paper and \cite{MR4768308} differ in several respects, both conceptually and technically\footnote{Although some results are phrased similarly, this is merely to facilitate comparison for the reader.}. Another point worth emphasizing is that this paper focuses exclusively on the theoretical analysis of convergence rates, and does not address specific computational or practical applications. These aspects are covered in works such as Das et al. \cite{MR3718733}, Sander and Meiss \cite{MR4104977,arXiv:2409.08496,MR4853295}, Meiss and Sander \cite{MR4322369,MR4845476}, Duignan and Meiss \cite{MR4582163}, among others. Very recently, the authors presented a comprehensive survey of these applied aspects in \cite{TL25c}; interested readers are referred to that work for further details\footnote{
	This paper appeared before \cite{TL25c}, but its new contributions do \textit{not} conflict with that work.
	Here we focus primarily on multiple Birkhoff averages; for the first time we achieve universal exponential convergence in the almost-periodic setting by a variety of methods, and simultaneously provide additional dynamical-systems insight.
	By contrast, \cite{TL25c} primarily aims to provide a comprehensive survey and to give practitioners a more user-friendly, quantitative account of weighted Birkhoff averages.}. Finally, we note that the results of this paper extend directly to more \textit{general} nonlinear dynamical systems, provided they are smoothly conjugated to quasi-periodic or almost-periodic ones; this has been detailed in \cite{MR3718733,MR3755876,MR4582163,TL25c}.

To finalize the Introduction, we organize this paper as follows. In Section \ref{Definitions and notations}, we provide  basic definitions and notations in both finite and infinite dimensional settings. Our main abstract results (Theorems \ref{MT1} to \ref{MT5}) regarding  arbitrary polynomial and exponential convergence for quasi-periodic and almost periodic cases are stated in Section \ref{Statement of results}, with their proofs postponed to Section \ref{Proof of main results}. Section \ref{Secexamples} presents joint Diophantine rotations as explicit examples to illustrate our results (Theorems \ref{MCORO1} to \ref{MCOCO4.7}). Furthermore, we establish the universality of exponential convergence under analytic observables and our accelerated weighting approach using truncated techniques in both finite and infinite dimensional settings (Theorems \ref{MCOCO4.4} and  \ref{MCOCO4.7}).  In Section \ref{Optimality of convergence rate},  we demonstrate  the optimality of this weighting method and the indispensability of our proposed nonresonant jointness in terms of preserving rapid convergence throughout this paper, by constructing counterexamples. Finally, in Section \ref{MULSEC6}, we provide numerical simulations and analysis to further illustrate and validate our results.

\section{Definitions and notations}\label{Definitions and notations}
To present the main results and applications, we require some definitions and notations that establish the foundation of our discussion.

For the convenience of later use, throughout this paper, $ \mathcal{O}( \cdot ) $, $ \mathcal{O}^{\#}( \cdot ) $ and $ o( \cdot ) $  are uniform with respect to $ N , T$ or $x $ sufficiently large, without causing ambiguity. We recall that for non-negative functions $ a(x) $ and $ b(x) $ with $ x>0 $ sufficiently large, $ a(x)=\mathcal{O}\left(b(x)\right) $ implies that there exists a universal constant $ \Theta_1>0 $ independent of $ x $ such that $ a(x) \leqslant \Theta_1 b(x) $,  $ a(x)=\mathcal{O}^{\#}\left(b(x)\right) $ implies that there exist  universal constants $ \Theta_2,\Theta_3>0 $ independent of $ x $ such that $ \Theta_2 b(x) \leqslant a(x) \leqslant \Theta_3 b(x) $, and finally $ a(x)=o\left(b(x)\right) $ implies that for any given $ \varepsilon>0 $ independent of $ x $, there holds $ a(x)\leqslant \varepsilon b(x) $. Denote by $ \left|  \cdot  \right| $ the sup-norm on the finite dimensional vector space $ \mathbb{R}^{d} $ with $ d \in \mathbb{N}^+ $ (or the infinite dimensional vector space $ \mathbb{R}^{\mathbb{N}} $).

In order to characterize the asymptotic behavior of nonresonance for rotations and Fourier coefficients of observables, we need to introduce the following approximation function.
\begin{definition}[Approximation function]\label{Approximationfundef}
	A function $ \Delta :\left[ {1, + \infty } \right) \to \left[ {1, + \infty } \right) $ is said to be an approximation function, if it is continuous, strictly monotonically  increasing, and satisfies  $ \Delta(+\infty) =+\infty$.
\end{definition}

\begin{definition}[Adaptive function]\label{AdaptivedEF}
	A function $ \varphi(x) $ defined on $ \left[ {1, + \infty } \right) $ is said to be an adaptive function, if it is nondecreasing, and satisfies that $ \varphi(+\infty)=+\infty $ and  $ \varphi \left( x \right) = o\left( x \right) $ as  $ x \to +\infty $.
\end{definition}
\begin{remark}
	For instance, both $ \varphi_1(x)=\log^{u}(1+x) $ with $ u>0 $ and $ {\varphi _2}\left( x \right) = {x^v} $ with $ 0<v<1 $ are adaptive functions. Carefully selecting a suitable adaptive function is crucial for achieving (and potentially enhancing) the exponential convergence rate of multiple ergodic averages, as we will detail later.
\end{remark}

We now introduce some concepts specific to the finite-dimensional case, where the space of variables is the torus  $ \mathbb{T}^d $ with $ 1 \leqslant d < +\infty $. In this case, the spatial structure is not essential,  since norms in finite-dimensional spaces are always equivalent. Denote by  $\left\|k\right\|=|k_1|+\cdots+|k_d|$ the $ 1 $-norm  for all $ k \in \mathbb{Z}^d $ throughout this paper. Next, we define the finite-dimensional analytic function space $ \mathcal{G}( {\mathbb{T}_\sigma ^d } ) $ as follows, which is well known to be a Banach space with the algebra property.

\begin{definition}[Finite-dimensional analyticity]\label{Finite-dimensional analyticity}
	For $ d \in \mathbb{N}^+ $ and $ \sigma>0 $, the thickened finite-dimensional torus $ \mathbb{T}_\sigma ^d $ is defined as
	\[\mathbb{T}_\sigma ^d : = \left\{ {\theta = {{({\theta_j})}_{1 \leqslant j \leqslant d}},\;{\theta_j} \in \mathbb{C}:\operatorname{Re} {\theta_j} \in \mathbb{T},\;\left| {\operatorname{Im} {\theta_j}} \right| \leqslant \sigma },\; 1 \leqslant j \leqslant d \right\}.\]
	Then the Banach space of analytic functions $ \mathcal{G}( {\mathbb{T}_\sigma ^d } ) $ is defined as 
	\[\mathcal{G}( {\mathbb{T}_\sigma ^d } ): = \left\{ {u\left(\theta  \right) = \sum\limits_{k  \in \mathbb{Z}^d } {\hat u_k{e^{2\pi i k  \cdot \theta }}} :\|u\|_{\sigma }: = \sum\limits_{k  \in \mathbb{Z} ^d } {\left| {\hat u_k} \right|{e^{2 \pi \sigma {{\| k  \|} }}}}  <  + \infty } \right\}.\]
\end{definition}

In contrast to the classical ($ 1 $-dimensional) ergodic case, we need to propose the concept of the \textit{joint} nonresonant condition for rotations in weighted multiple averages $ \mathrm{DMW}_N^\ell $ and $ \mathrm{CMW}_T^\ell $. As we will demonstrate later through a counterexample in Section \ref{Optimality of convergence rate}, it is important to note that \textit{such nonresonant jointness is indispensable and cannot be removed}.

\begin{definition}[Finite-dimensional joint  nonresonant condition] \label{Fi}
	The rotational vectors $  \{\rho_j\}_{j=1}^\ell  \in \mathbb{T}^d $ are said to satisfy the  Finite-dimensional joint nonresonant condition, if there exist  $ \alpha >0 $ and an approximation function $ \Delta $, such that the joint rotation $\tilde{\rho}:=\left(\rho_1,\ldots,\rho_\ell\right)\in \mathbb{T}^{d\ell}$ satisfies:
	\begin{itemize}

		\item[($ a $)] The discrete case
		\begin{equation}\label{nonresonant}
			\left| {k \cdot \tilde{\rho}  - n} \right| \geqslant \frac{\alpha }{{\Delta \left( {||k||} \right)}},\;\;\forall 0 \ne k \in {\mathbb{Z}^{d\ell}},\;\;\forall n \in \mathbb{Z};
		\end{equation}
		\item[($ b $)] The continuous case
		\begin{equation}\label{nonresonant2}
			\left| {k \cdot \tilde{\rho} } \right| \geqslant \frac{\alpha }{{\Delta \left( {||k||} \right)}},\;\;\forall 0 \ne k \in {\mathbb{Z}^{d\ell}}.
		\end{equation}
	\end{itemize}
\end{definition}
It is evident that every irrational vector  $ \rho $ can be associated with an approximation function $ \Delta $ in the form of nonresonance  \eqref{nonresonant} (or \eqref{nonresonant2}). For instance, a well-defined function $ \Delta \left( x \right): = {\max _{0 < \left\| k \right\| \leqslant x}}{\left( {{\rm dist}\left( {k \cdot \rho ,\mathbb{Z}} \right)} \right)^{ - 1}} $  (or $ \Delta \left( x \right): = {\max _{0 < \left\| k \right\| \leqslant x}}{\left| {k \cdot \rho } \right|^{ - 1}} $) is sufficient  (the value of $ \Delta $ can be adjusted to ensure strict monotonic increase). A commonly encountered scenario is the Diophantine type, as presented below.

\begin{definition}[Finite-dimensional joint  Diophantine condition] \label{Finite-dimensional Diophantine condition}
	We say that rotational vectors $  \{\rho_j\}_{j=1}^\ell  \in \mathbb{T}^d $ satisfy the Finite-dimensional joint Diophantine condition, if the approximation function in Definition \ref{Fi} is $ \Delta(x)=x^\tau $ with  $\tau > d\ell  $ in \eqref{nonresonant} and with $\tau > d\ell-1  $ in \eqref{nonresonant2}.
\end{definition}
\begin{remark}\label{Remarkprobability}
	It is well known that the above joint Diophantine vectors form a set of full Lebesgue measure. See, for instance, Herman \cite{MR538680}. Therefore, the assumption that the rotations  are joint Diophantine is robust in a measure theoretic sense, i.e., in physical experiments the rotations will be joint Diophantine with probability $ 1 $.
\end{remark}

However, when considering the infinite-dimensional torus $ {\mathbb{T}^\infty }: = {\mathbb{T}^\mathbb{N}} $ ($ d=\infty $), it becomes necessary to impose some spatial structure (which may not be unique) to prevent the Fourier series expansions from blowing up. For convenience, we use the Diophantine condition for irrational vectors   proposed by Bourgain and the associated  spatial structure, as detailed in  \cite{MR2180074,MR4201442}. More precisely, our set of rotational vectors is the infinite-dimensional cube $ {\left[ {1,2} \right]^\mathbb{N}} $ (equivalent  to $ {\mathbb{T}^\infty } $), endowed  with the probability measure induced by the product measure of the cube $ {\left[ {1,2} \right]^\mathbb{N}} $. Subsequently, for fixed $ 2 \leqslant \eta \in \mathbb{N}^+ $, we define the following set of infinite integer vectors with finite support:
\begin{equation}\notag
	\mathbb{Z}_ * ^\infty : = \left\{ {k \in {\mathbb{Z}^\mathbb{N}}:{{\left| k \right|}_\eta }: = \sum\limits_{j \in \mathbb{N}} {{{\left\langle j \right\rangle }^\eta }\left| {{k_j}} \right|}  <  + \infty ,\;\;\left\langle j \right\rangle : = \max \left\{ {1,\left| j \right|} \right\}} \right\}.
\end{equation}
In this case,  it is evident that  $ {k_j} \ne 0 $ only for finitely many indices $ j \in \mathbb{N} $ when $k \in  	\mathbb{Z}_ * ^\infty  $ is fixed. It can be seen later that such a metric like $ {{{\left| \cdot \right|}_\eta }} $ is necessary for the infinite-dimensional case since it determines the boundedness of the summation used in the  proof. It should be pointed out that for rotations on $ \mathbb{T}^\infty $, one could establish more general assumptions (analogous to the boundedness conditions \eqref{jointcon1}, \eqref{jointcon2}, and the truncated smallness  conditions \eqref{T3jointcon1}, \eqref{T4jointcon1}) and corresponding theorems based on  rotations with full probability measure, such as the almost critical nonresonant conditions in \cite{MR538680} (note that a certain criticality helps to weaken the decay requirement for the Fourier coefficients of observables $ F_j $ with $ 1 \leqslant j \leqslant \ell $). However, for the sake of simplicity, we choose not to explicitly state them here.

Given the aforementioned spatial structure, let us now introduce the infinite-dimensional analytic function space $ \mathcal{G}\left( {\mathbb{T}_\sigma ^\infty } \right) $, defined as follows, which has been shown to be a Banach space with the algebra property, as demonstrated in \cite{MR4201442}, among others.

\begin{definition}[Infinite-dimensional analyticity]\label{Infinite-dimensional analyticity}
	For $ 2 \leqslant \eta \in \mathbb{N}^+ $ and $ \sigma>0 $,  the thickened infinite-dimensional torus $ \mathbb{T}_\sigma ^\infty $ is defined as 
	\[\mathbb{T}_\sigma ^\infty : = \left\{ {\theta = {{({\theta_j})}_{j \in \mathbb{N}}},\;{\theta_j} \in \mathbb{C}:\operatorname{Re} {\theta_j} \in \mathbb{T},\;\left| {\operatorname{Im} {\theta_j}} \right| \leqslant \sigma {{\left\langle j \right\rangle }^\eta }},\; j \in \mathbb{N} \right\}.\]
	Then the Banach space of analytic functions $ \mathcal{G}\left( {\mathbb{T}_\sigma ^\infty } \right) $ is given by
	\[\mathcal{G}( {\mathbb{T}_\sigma ^\infty } ): = \left\{ {u\left(\theta  \right) = \sum\limits_{k  \in \mathbb{Z}_ * ^\infty } {\hat u_k{e^{2\pi i k  \cdot \theta }}} :\|u\|_{\sigma }: = \sum\limits_{k  \in \mathbb{Z}_ * ^\infty } {\left| {\hat u_k} \right|{e^{2 \pi \sigma {{\left| k  \right|}_\eta }}}}  <  + \infty } \right\}.\]
\end{definition}

Similar to Definition \ref{Fi}, we introduce the infinite-dimensional versions of the joint nonresonant conditions in Definitions \ref{In} and \ref{wuqiongdio}.

\begin{definition}[Infinite-dimensional joint  nonresonant condition]\label{In}
Let $ 2 \leqslant \eta \in \mathbb{N}^+ $ be given.	The rotational vectors $ \{\rho_j\}_{j=1}^\ell  \in \mathbb{T}^{\infty} $  are said to satisfy the Infinite-dimensional joint nonresonant condition, if there exist  $ \gamma >0 $ and an approximation function $ \vartheta $, such that the joint rotation $\tilde{\rho}:=\left(\rho_1,\ldots,\rho_\ell\right)\in \mathbb{T}^{\infty}$ satisfies:
	\begin{itemize}

		\item[($ c $)] The discrete case
		\begin{equation}\label{Infinite-dimensional nonresonant condition}
			\left| {k \cdot \tilde\rho  - n} \right| > \frac{\gamma }{{{\vartheta}( {{{\left| k \right|}_\eta }} )}},\;\;\forall 0 \ne k \in \mathbb{Z}_ * ^\infty ,\;\;\forall n \in \mathbb{Z};
		\end{equation}
		
		\item[($ d $)] The continuous case
		\begin{equation}\label{Infinite-dimensional nonresonant condition2}
			\left| {k \cdot \tilde\rho } \right| > \frac{\gamma }{{{\vartheta}( {{{\left| k \right|}_\eta }} )}},\;\;\forall 0 \ne k \in \mathbb{Z}_ * ^\infty .
		\end{equation}
	\end{itemize}	
\end{definition}

\begin{definition}[Infinite-dimensional joint  Diophantine condition] \label{wuqiongdio}
	We say that rotational vectors $ \{\rho_j\}_{j=1}^\ell  \in \mathbb{T}^{\infty} $ satisfy the Infinite-dimensional joint Diophantine condition, if the approximation function in Definition \ref{In} satisfies 	\begin{equation}\notag
		{\vartheta}( {{{\left| k \right|}_\eta }} ) = \prod\limits_{j \in \mathbb{N}} {\left( {1 + {{\left| {{k_j}} \right|}^\mu }{{\left\langle j \right\rangle }^\mu }} \right)} ,\;\;\forall 0 \ne k \in \mathbb{Z}_ * ^\infty
	\end{equation}
	with some $ \mu >1 $.
\end{definition}
\begin{remark}\label{REMARKDIO}
	Denote by $ \mathscr{D}_{\gamma ,\mu } $ the set of all vectors satisfying the infinite-dimensional joint Diophantine condition defined in Definition \ref{wuqiongdio}. It can be shown that $ \mathscr{D}_{\gamma ,\mu } $ has full probability measure, indicating that such rotations are universal. For further details, see  \cite{MR4091501,MR2180074}.
\end{remark}


\begin{definition}[Spaces of rapid convergence] \label{Spaces}
	Assume $ \left( {\mathscr{B},|| \cdot ||{_\mathscr{B}}} \right) $ is a Banach function space with the algebra property (may be infinite-dimensional).
	
	For the finite-dimensional case, let  $ f:{\mathbb{T}^d} \to \mathscr{B} $ with
	\begin{equation}\label{ff}
		f(\theta) = \sum\limits_{ k \in {\mathbb{Z}^d}} {{{\hat f}_k}{e^{2\pi ik \cdot \theta }}} ,\;\;{{\hat f}_k} = \int_{{\mathbb{T}^d}} {f( {\hat \theta } ){e^{ - 2\pi ik \cdot \hat \theta }}d\hat \theta },
	\end{equation}
	where the first ``$ = $" represents equality in the norm $ || \cdot ||{_\mathscr{B}} $.
	Now, define the following  space 
	\begin{equation}\label{Bdelta}
		{\mathscr{B}_{\tilde \Delta }}: = \left\{ {f:{\mathbb{T}^d} \to \mathscr{B} :\text{$ f  $ satisfies \eqref{ff}, $ \mathop {\sup }\limits_{0 \ne k \in {\mathbb{Z}^d}} \tilde \Delta \left( {||k||} \right)\| {{{\hat f}_k}} \|_{_\mathscr{B}} <  + \infty $} } \right\}
	\end{equation}
	for a given approximation function $ {\tilde \Delta } $. For $ f,g \in \mathscr{B} $,  define the product $ fg\in \mathscr{B} $ (due to the algebra property) as
	\begin{equation}\label{product}
		(fg)(\theta): = \sum\limits_{k,j \in {\mathbb{Z}^d}} {{{\hat f}_k}{{\hat g}_j}{e^{2\pi i(k+j) \cdot \theta }}} \in \mathscr{B} .
	\end{equation}
	\begin{remark}
	If the observables are vector-valued functions, we can similarly consider the Hadamard product.
	\end{remark}
	As for the infinite-dimensional case, we replace $ \mathbb{T}^d $ with $ \mathbb{T}^{\infty} $, the approximation function $ \tilde\Delta $ with $ {\tilde \Delta }_\infty $, the metric $ \|\cdot\| $ with $ |\cdot|_{\eta} $,   $ k\in \mathbb{Z}^d $ with $ k\in \mathbb{Z}_*^{\infty} $ in \eqref{ff} and \eqref{Bdelta} for distinction, and denote by $ \mathscr{B}_{\tilde \Delta_\infty } $ the corresponding function space with rapid convergence. We also refer to  \cite{MR4201442} for discussions on integrals and Fourier expansions on the infinite-dimensional torus $ \mathbb{T}^\infty $.
	
	Additionally, in the case where $ f $ is a trigonometric polynomial of order $ K \in \mathbb{N}^+ $, we propose the following spaces, respectively:
	\[	{\mathscr{B}_{\tilde \Delta ,K}}: = \left\{ {f \in {\mathscr{B}_{\tilde \Delta }}:{\hat{f}_k} = 0 \text{ for all } \left\| k \right\| > K \in {\mathbb{N}^ + }} \right\},\]
	and
	\[	{\mathscr{B}_{\tilde \Delta_\infty ,K}}: = \left\{ {f \in {\mathscr{B}_{\tilde \Delta_\infty }}:{\hat{f}_k} = 0 \text{ for all } | k  |_\eta > K \in {\mathbb{N}^ + }} \right\}.\]
\end{definition}

\section{The statement of the abstract main  results}\label{Statement of results}
Consider the multiple ergodic averages $ \mathrm{DMW}_N^\ell $ in \eqref{DMW} and $ \mathrm{CMW}_N^\ell $ in \eqref{CMW}. In the following discussion,  we always assume that the observables $ F_j \in \mathscr{B}_{\tilde \Delta_j } $ (or $ F_j \in \mathscr{B}_{\tilde \Delta_{\infty j} } $) with approximation functions $ \tilde{\Delta}_j $ (or $ \tilde{\Delta}_{\infty j}  $), where $ 1 \leqslant j \leqslant \ell $, and the  rotations $ \{\rho_j\}_{j=1}^\ell  \in \mathbb{T}^{d} $ satisfy the Finite-dimensional joint  nonresonant condition in Definition \ref{Fi} with an approximation function $ \Delta $ (or $ \{\rho_j\}_{j=1}^\ell  \in \mathbb{T}^{\infty} $ satisfy the Infinite-dimensional joint  nonresonant condition in Definition \ref{In} with an approximation function $ \vartheta $), i.e., the joint rotation $\tilde{\rho}=\left(\rho_1,\ldots,\rho_\ell\right)\in \mathbb{T}^{d\ell}$ (or $ \mathbb{T}^\infty $) is of nonresonant type with $ \Delta $ (or $ \vartheta $).  Accordingly, denote by $ \sum\nolimits_{j = 1}^\ell  {{k^j}} : = \left( {{k^1}, \ldots ,{k^\ell }} \right) \in {\mathbb{Z}^{d\ell }} $ (or $ \mathbb{Z}_*^\infty $) the \textit{joint integer vector} of $ \{k^j\}_{j=1}^{\ell} \in \mathbb{Z}^d $ (or $ \mathbb{Z}_*^\infty $) for convenience. Now, we are in a position to present our abstract main  results concerning  rapid convergence on weighted multiple ergodic averages, specifically addressing the convergence rate of arbitrary polynomial and exponential types based on different assumptions, respectively. Explicit situations will be postponed to Section \ref{Secexamples}.

\subsection{Arbitrary polynomial convergence}\label{Arbitrary polynomial convergence3}
In order to establish the polynomial convergence of the multiple ergodic averages $ \mathrm{DMW}_N^\ell $  and $ \mathrm{CMW}_T^\ell $,  it is essential to introduce specific \textit{boundedness conditions} (which can also be referred to as balancing conditions) regarding the joint nonresonant properties of rotations and Fourier coefficients of observables. These conditions are explicitly defined by \eqref{jointcon1} for the finite-dimensional case and by \eqref{jointcon2} for the infinite-dimensional case.  We point out that, although they appear complicated in form, they are actually quite \textit{typical} (or \textit{weak}) conditions that allow very general nonresonance and regularity, as detailed in Section \ref{Secexamples}. It is worth noting that, when $ 2 \leqslant m \in \mathbb{N}^+ $ is allowed to be  fixed arbitrarily, the resulting convergence rate will exhibit an \textit{arbitrary} polynomial type.

Our Theorems \ref{MT1} and \ref{MT2} concerning arbitrary polynomial convergence, can be summarized as follows:

\begin{theorem}[Arbitrary polynomial convergence in the finite-dimensional case]\label{MT1}
	Consider the quasi-periodic case. Assume that the approximation functions $ \Delta$, $ \tilde \Delta_j (1 \leqslant j \leqslant \ell) $ satisfy  the boundedness condition with some $ 2 \leqslant m  \in \mathbb{N}^+ $:
	\begin{equation}\label{jointcon1}
		\sum\limits_{0 \ne {k^1}, \ldots ,{k^\ell } \in {\mathbb{Z}^d}} {\frac{{{\Delta ^m}\left( { \sum\nolimits_{j = 1}^\ell  {\left\| {{k^j}} \right\|}  } \right)}}{{\prod\nolimits_{j = 1}^\ell  {{{\tilde \Delta }_j}\left( {\left\| {{k^j}} \right\|} \right)} }}}  <  + \infty .
	\end{equation}
	Then $ \mathrm{DMW}_N^\ell $ and $ \mathrm{CMW}_T^\ell $ exhibit  polynomial convergence, i.e.,
	\begin{equation}\label{MT1lisan}
		{\left\| {\mathrm{DMW}_N^\ell\left( \mathcal{F} \right)\left( \theta  \right) - \prod\limits_{j = 1}^\ell  \left({\int_{{\mathbb{T}^d}} {{F_j}(\hat \theta )d\hat \theta } }\right) } \right\|_\mathscr{B}} = \mathcal{O}\left( {{N^{ - m}}} \right),
	\end{equation}
	and
	\begin{equation}\label{MT1lianxu}
		{\left\| {\mathrm{CMW}_T^\ell\left( \mathcal{F} \right)\left( \theta  \right) - \prod\limits_{j = 1}^\ell  \left({\int_{{\mathbb{T}^d}} {{F_j}(\hat \theta )d\hat \theta } }\right) } \right\|_\mathscr{B}} = \mathcal{O}\left( {{T^{ - m}}} \right),
	\end{equation}
	whenever $ N$ and $T $ are sufficiently large.
\end{theorem}

\begin{theorem}[Arbitrary polynomial convergence in the infinite-dimensional case]\label{MT2}
	Consider the almost periodic case. Assume that the approximation functions $ \vartheta $, $ \tilde \Delta _{\infty j} (1 \leqslant j \leqslant \ell) $ satisfy  the boundedness condition with some $ 2 \leqslant m  \in \mathbb{N}^+ $:
	\begin{equation}\label{jointcon2}
		\sum\limits_{0 \ne {k^1}, \ldots ,{k^\ell } \in \mathbb{Z}_ * ^\infty } {\frac{{{{\vartheta} ^m}\left( {{\sum\nolimits_{j = 1}^\ell  {{{\left| {{k^j}} \right|}_\eta }}  }} \right)}}{{\prod\nolimits_{j = 1}^\ell  {{{\tilde \Delta }_{\infty j}}\left( {{{\left| {{k^j}} \right|}_\eta }} \right)} }}}  <  + \infty .
	\end{equation}
	Then $ \mathrm{DMW}_N^\ell $ and $ \mathrm{CMW}_T^\ell $ exhibit  polynomial convergence, i.e.,
	\begin{equation}\label{MT2-1}
		{\left\| {\mathrm{DMW}_N^\ell\left( \mathcal{F} \right)\left( \theta  \right) - \prod\limits_{j = 1}^\ell  \left({\int_{{\mathbb{T}^\infty}} {{F_j}(\hat \theta )d\hat \theta } }\right) } \right\|_\mathscr{B}} = \mathcal{O}\left( {{N^{ - m}}} \right),
	\end{equation}
	and
	\begin{equation}\label{MT2-2}
		{\left\| {\mathrm{CMW}_T^\ell\left( \mathcal{F} \right)\left( \theta  \right) - \prod\limits_{j = 1}^\ell  \left({\int_{{\mathbb{T}^\infty}} {{F_j}(\hat \theta )d\hat \theta } }\right) } \right\|_\mathscr{B}} = \mathcal{O}\left( {{T^{ - m}}} \right),
	\end{equation}
	whenever $ N$ and $T $ are sufficiently large.
\end{theorem}

To illustrate the results more clearly, let us make some comments below.  In contrast to the analysis of a single observable as discussed by the authors in \cite{MR4768308}, the jointness of the  boundedness conditions \eqref{jointcon1} and \eqref{jointcon2} is crucial in our context, because we are dealing with \textit{multiple} ergodic averages throughout this paper. These conditions are indispensable for achieving rapid convergence of the form $ \mathcal{O}(N^{-m}) $, which is distinct from the slower convergence rate of $ \mathcal{O}(N^{-1}) $ observed in classical ergodic theory. Such  existence of the boundedness in \eqref{jointcon1} and \eqref{jointcon2} is evident when viewed from the perspective of the L'Hopital's rule. It could be naturally guaranteed  for the  given joint nonresonance of rotations $ \{\rho_j\}_{j=1}^{\ell} $, i.e., the approximation function $ \Delta $ (or $ \vartheta $), whenever $ \{\tilde{\Delta}_j\}_{j=1}^{\ell} $ (or $ \{\tilde{\Delta}_{\infty j}\}_{j=1}^{\ell} $) are \textit{relatively} large enough.  Additionally, there exist many cases such that the boundedness conditions hold for any fixed $ 2 \leqslant m \in \mathbb{N}^+ $, e.g., $ \Delta $ is of polynomial type (Diophantine irrationality) while $ \{\tilde{\Delta}_j\}_{j=1}^{\ell} $ are of exponential type (analyticity for $ \{F_j\}_{j=1}^{\ell} $), as shown in Section \ref{Secexamples}.  Recall the universality explained in  Remark \ref{Remarkprobability}. Therefore, as a conclusion, $ \mathrm{DMW}_N^\ell $ and $ \mathrm{CMW}_T^\ell $ can always achieve arbitrary polynomial convergence  for the majority of physical problems, when the observables $ \{F_j\}_{j=1}^{\ell} $ under consideration are always sufficiently smooth.  Thus,  our weighting method demonstrates  an excellent acceleration effect  even in the context of multiple ergodic averages.

\subsection{Exponential convergence}\label{SUBExponential convergence}
The previous results regarding arbitrary polynomial convergence in  weighted multiple ergodic averages $ \mathrm{DMW}_N^\ell $ and $ \mathrm{CMW}_T^\ell $ naturally prompt the following fundamental questions:
\begin{itemize}
	\item \textit{Can exponential convergence be achieved in quasi-periodic and almost periodic settings? If so, to what forms should the boundedness conditions be strengthened? Is exponential convergence a universal phenomenon?}
\end{itemize}

These questions are quite non-trivial, particularly  the last one, and they hold significant importance in both  theory and computation. We will comprehensively address them in this section.  As discussed in Section \ref{SubsecProT1}, the universal control constants omitted in Theorems \ref{MT1} and \ref{MT2} depend on $ m $; more preciously, they tend to $ +\infty $ as $ m \to +\infty $. However, exponential convergence can indeed be achieved in the simplest continuous case, where we consider  $ F_1(x)=\sin (2\pi x) $ and $ F_j (x)=1 $ for $ 2 \leqslant j \leqslant \ell $. In this case, the multiple ergodic average $ \mathrm{CMW}_T^1 $ is reduced to the generic case in $ \mathbb{R}^1 $:
\[	\mathrm{CMW}_T^1(\mathcal{F})\left( \theta  \right): = \frac{1}{T}\int_0^T {w\left( {s/T} \right)\sin\left( 2\pi(\theta + s \rho) \right) ds},\]
which has been analyzed  in \cite{MR4768308}. Actually, by utilizing integration by parts, Lemma \ref{GAODAO} in the Appendix, and conducting specific asymptotic analysis, one can obtain  exponential convergence for it.  Therefore, subject to specific assumptions, there is reason to believe that exponential convergence is attainable for the weighted multiple averages $ \mathrm{DMW}_N^\ell $ and $ \mathrm{CMW}_T^\ell $. However, it is worth mentioning that the corresponding joint assumptions, referred to as \textit{truncated smallness conditions} \eqref{T3jointcon1} and \eqref{T4jointcon1} below, are much more complicated due to the presence of \textit{multiplicity} and act like \eqref{jointcon1} and \eqref{jointcon2}. Furthermore, in order to establish \textit{universality} for exponential convergence over $ \mathbb{T}^\infty $, an innovative truncation technique must be introduced to address the challenges posed by general observables (infinite trigonometric series).

Let an adaptive function $ \varphi(x) $ be given. For the sake of simplicity, we first define the \textit{truncated spaces} for both quasi-periodic and almost periodic  cases as:
\begin{equation}\label{quasiS}
	\mathscr{S}\left( x \right): = \left\{ {\sum\nolimits_{j = 1}^\ell  {{k^j}}  \in {\mathbb{Z}^{d\ell }}:0 \ne {k^j} \in {\mathbb{Z}^d},\left\| {{k^j}} \right\| \leqslant \ell^{-1} {\Delta ^{ - 1}}\left( {x/\varphi \left( x \right)} \right)},1\leqslant j \leqslant \ell \right\},
\end{equation}
and
\begin{equation}\label{almostS}
	{\mathscr{S}_\infty }\left( x \right): = \left\{ {\sum\nolimits_{j = 1}^\ell  {{k^j}}  \in \mathbb{Z}_ * ^\infty :0 \ne {k^j} \in \mathbb{Z}_ * ^\infty ,{{\left| {{k^j}} \right|}_\eta } \leqslant \ell^{-1} {{\vartheta} ^{ - 1}}\left( {x/\varphi \left( x \right)} \right)},1\leqslant j \leqslant \ell \right\}.
\end{equation}

Now, we are in a position to present our exponential convergence theorems  that involve both quasi-periodicity and almost periodicity, namely Theorems \ref{MT3} and \ref{MT4}, respectively.

\begin{theorem}[Exponential convergence in the  finite-dimensional case]\label{MT3}
	Consider the quasi-periodic case. Assume that the approximation functions $ \Delta$, $ \tilde \Delta_j (1 \leqslant j \leqslant \ell) $ satisfy the   truncated smallness  condition with some $ c>0 $:
	\begin{equation}\label{T3jointcon1}
		\sum\limits_{0 \ne \sum\nolimits_{j = 1}^\ell  {{k^j}}  \in {\mathbb{Z}^{d\ell }}\backslash \mathscr{S}\left( x \right)} {\frac{1}{{\prod\nolimits_{j = 1}^\ell  {{{\tilde \Delta }_j}\left( {\left\| {{k^j}} \right\|} \right)} }}}  = \mathcal{O}\left( {{e^{ - cx}}} \right).
	\end{equation}
	 Then there exists an absolute constant $ \beta  _1>0 $ such that $ \mathrm{DMW}_N^\ell $ and $ \mathrm{CMW}_T^\ell $ exhibit  exponential  convergence with $ N,T $ sufficiently large, i.e.,
	\begin{equation}\label{MT3lisan}
		{\left\| {\mathrm{DMW}_N^\ell\left( \mathcal{F} \right)\left( \theta  \right) - \prod\limits_{j = 1}^\ell  \left({\int_{{\mathbb{T}^d}} {{F_j}(\hat \theta )d\hat \theta } }\right) } \right\|_\mathscr{B}} =  \mathcal{O}\left( \exp(-\varphi^{\beta_1}(N)) \right),
	\end{equation}
	and
	\begin{equation}\label{MT3lianxu}
		{\left\| {\mathrm{CMW}_T^\ell\left( \mathcal{F} \right)\left( \theta  \right) - \prod\limits_{j = 1}^\ell  \left({\int_{{\mathbb{T}^d}} {{F_j}(\hat \theta )d\hat \theta } }\right) } \right\|_\mathscr{B}} = \mathcal{O}\left( \exp(-\varphi^{\beta_1}(T)) \right).
	\end{equation}
\end{theorem}

\begin{theorem}[Exponential convergence in the  infinite-dimensional case]\label{MT4}
	Consider the almost periodic case. Assume that the approximation functions $ \vartheta $, $ \tilde \Delta _{\infty j} (1 \leqslant j \leqslant \ell) $ satisfy  the   truncated smallness condition with some $ c>0 $:
	\begin{equation}\label{T4jointcon1}
		\sum\limits_{0 \ne \sum\nolimits_{j = 1}^\ell  {{k^j}}  \in {{\mathbb{Z}_ * ^\infty }}\backslash {\mathscr{S}_\infty }\left( x \right)} {\frac{1}{{\prod\nolimits_{j = 1}^\ell  {{{\tilde \Delta }_{\infty j}}\left( {{{\left| {{k^j}} \right|}_\eta }} \right)} }}}  = \mathcal{O}\left( {{e^{ - cx}}} \right).
	\end{equation}
	  Then there exists an absolute constant $ \beta  _2>0 $ such that $ \mathrm{DMW}_N^\ell $ and $ \mathrm{CMW}_T^\ell $ exhibit exponential  convergence with $ N,T $ sufficiently large, i.e.,
	\begin{equation}\notag
		{\left\| {\mathrm{DMW}_N^\ell\left( \mathcal{F} \right)\left( \theta  \right) - \prod\limits_{j = 1}^\ell  \left({\int_{{\mathbb{T}^\infty}} {{F_j}(\hat \theta )d\hat \theta } }\right) } \right\|_\mathscr{B}} =  \mathcal{O}\left( \exp(-\varphi^{\beta_2}(N)) \right),
	\end{equation}
	and
	\begin{equation}\notag
		{\left\| {\mathrm{CMW}_T^\ell\left( \mathcal{F} \right)\left( \theta  \right) - \prod\limits_{j = 1}^\ell  \left({\int_{{\mathbb{T}^\infty}} {{F_j}(\hat \theta )d\hat \theta } }\right) } \right\|_\mathscr{B}} = \mathcal{O}\left( \exp(-\varphi^{\beta_2}(T)) \right).
	\end{equation}
\end{theorem}

Let us make some comments on our main Theorems \ref{MT3} and \ref{MT4} regarding  exponential convergence.
\begin{itemize}
	\item[\textbf{(C1)}] Note  we require that the approximation function $ \Delta(x) $ is strictly increasing with $ \Delta(+\infty)=+\infty $ in Definition \ref{Approximationfundef} and that the adaptive function $ \varphi(x) $ satisfies $  \varphi(+\infty)=+\infty $ in Definition \ref{AdaptivedEF}. Then the truncated spaces $ \mathscr{S}(x) $ in \eqref{quasiS} and $ \mathscr{S}_\infty(x) $ in \eqref{almostS} could approach the entire  spaces, namely $ {\mathbb{Z}^{d\ell }}\backslash \left\{ 0 \right\} $ and $ \mathbb{Z}_ * ^\infty \backslash \left\{ 0 \right\} $, whenever $ x \to +\infty $.  Therefore, the truncated smallness conditions \eqref{T3jointcon1} and \eqref{T4jointcon1} are reasonable (and represent stronger versions of the prior boundedness conditions \eqref{jointcon1} and \eqref{jointcon2}), i.e., the series
	\[\sum\limits_{0 \ne {k^1}, \ldots ,{k^\ell } \in {\mathbb{Z}^d}} {\frac{1}{{\prod\nolimits_{j = 1}^\ell  {{{\tilde \Delta }_j}\left( {\left\| {{k^j}} \right\|} \right)} }}} ,\;\;\sum\limits_{0 \ne {k^1}, \ldots ,{k^\ell } \in \mathbb{Z}_ * ^\infty } {\frac{1}{{\prod\nolimits_{j = 1}^\ell  {{{\tilde \Delta }_{\infty j}}\left( {{{\left| {{k^j}} \right|}_\eta }} \right)} }}} \]
	need to  converge rapidly  at a certain rate. It is evident   that they are always achievable provided that the Fourier coefficients of $ F_j $ decay rapidly enough for $ 1 \leqslant j \leqslant \ell $, similar to the  comments given in Section \ref{Arbitrary polynomial convergence3}.  \textit{It is important to highlight that \eqref{T3jointcon1} and \eqref{T4jointcon1} can be further weakened while  preserving exponential convergence, as demonstrated in Theorems \ref{MCOCO4.4} and \ref{MCOCO4.7}.}
	
	\item[\textbf{(C2)}] The resulting convergence rate in the aforementioned  theorems can  indeed be exponential, provided that the adaptive function is large enough. For instance, $ \varphi(x)\sim x^{\nu} $ with some $ \nu \in (0,1) $, or even $ \varphi \left( x \right) \sim x{\left( {\log x} \right)^{ - \kappa }} $ with some $ \kappa>0 $. However, it is important to note that the restriction $ \varphi(x)=o(x) $ cannot be removed; otherwise the truncated spaces $ \mathscr{S}(x) $  and $ \mathscr{S}_\infty(x) $ would not tend to  the entire  spaces. \textit{When considering  specific examples, the selection of an appropriate adaptive function is crucial.}
	
	\item[\textbf{(C3)}] 
	While the proof in the infinite-dimensional case follows a similar approach to the finite-dimensional case, our method effectively circumvents the \textit{Curse of Dimensionality} and achieves rapid exponential convergence, as the universal control constant remains dimension-independent, all due to the truncated smallness condition \eqref{T4jointcon1} that we have proposed.
\end{itemize}

In the proofs of Theorems \ref{MT1} to \ref{MT4}, small divisors arise due to integration by parts. This not only complicates the proof but also necessitates additional assumptions such as the boundedness conditions \eqref{jointcon1}, \eqref{jointcon2}, and the truncated smallness conditions \eqref{T3jointcon1}, \eqref{T4jointcon1}, which depend on the adaptive function $ \varphi(x) $. However, in the absence of small divisors, these challenges can be avoided, leading to an exponential convergence rate without the need for introducing the adaptive function, as demonstrated in Theorem \ref{MT5} below.

\begin{theorem}\label{MT5}
	Assume $ F_j \in {\mathscr{B}_{\tilde \Delta_j ,K}} $ (or $F_j \in  \mathscr{B}_{\tilde \Delta_{\infty j} ,K}  $) with $ 1 \leqslant j \leqslant \ell $, and let $ \{\rho_j\}_{j=1}^{\ell} $ satisfy the joint nonresonant condition in Definition \ref{Fi} (or \ref{In}). Then there exists some $ \hat{c}>0 $ such that the followings hold for $ N,T $ sufficiently large:
	\begin{equation}\label{MT51}
		{\left\| {\mathrm{DMW}_N^\ell\left( \mathcal{F} \right)\left( \theta  \right) - \prod\limits_{j = 1}^\ell  \left({\int_{{\mathbb{T}^d}} {{F_j}(\hat \theta )d\hat \theta } }\right) } \right\|_\mathscr{B}} = \mathcal{O}\left( \exp({-N^{\hat{c}}}) \right),\;\;1\leqslant d \leqslant \infty,
	\end{equation}
	and
	\begin{equation}\label{MT52}
		{\left\| {\mathrm{CMW}_T^\ell\left( \mathcal{F} \right)\left( \theta  \right) - \prod\limits_{j = 1}^\ell  \left({\int_{{\mathbb{T}^d}} {{F_j}(\hat \theta )d\hat \theta } }\right) } \right\|_\mathscr{B}} = \mathcal{O}\left( \exp({-T^{\hat{c}}}) \right),\;\;2\leqslant d \leqslant \infty.
	\end{equation}
	In the continuous case where $ d=1 $, the following holds for $ T $ sufficiently large:
	\begin{equation}\label{MT53}
		{\left\| {\mathrm{CMW}_T^\ell\left( \mathcal{F} \right)\left( \theta  \right) - \prod\limits_{j = 1}^\ell  \left({\int_{{\mathbb{T}^1}} {{F_j}(\hat \theta )d\hat \theta } }\right) } \right\|_\mathscr{B}} = \mathcal{O}\left( \exp({-T^{\hat{c}}}) \right),
	\end{equation}
	provided that  $ F_j \in {\mathscr{B}_{\tilde \Delta_j }} $ with $ 1 \leqslant j \leqslant \ell $ and satisfy
	\begin{equation}\label{MT54}
		\prod\limits_{j = 1}^\ell  {\left( {\int_1^{ + \infty } {\frac{1}{{{{\tilde \Delta }_j}\left( x \right)}}dx} } \right)}  <  + \infty ,
	\end{equation}
	and the finite nonresonant conditions $\sum\nolimits_{j = 1}^\ell  {{k^j}{\rho _j}}  \ne 0$  hold for $0 < \left\| {{k^j}} \right\| \leqslant K$.
\end{theorem}
\begin{remark}
	In other words, the finite nonresonance in the continuous case where $d=1$ is not genuine  nonresonance, meaning that  $ \{\rho_j\}_{j=1}^{\ell} $ could be resonant at distances much greater than $  K$. This fundamental distinction between the continuous and discrete cases is worth noting.
\end{remark}

\section{Explicit applications via quasi-periodicity and almost periodicity}\label{Secexamples}
To illustrate the practical implications enabled by the abstract theorems in this paper, we present five variants concerning  (arbitrary) polynomial convergence and exponential convergence  in the context of multiple types, involving both finite and infinite dimensional cases. As the \textit{essential} conclusion of this paper, Theorem \ref{MCOCO4.7} demonstrates that the  weighted multiple ergodic averages $ 	\mathrm{DMW}_N^\ell $ and $ \mathrm{CMW}_T^\ell $ exhibit universal exponential pointwise convergence when the observables are analytic. Prior to this, we establish their relationship with KAM theory.

The weighting method with \eqref{expfun} for $ 1 $-dimensional Birkhoff averages has been utilized in computational applications within KAM theory. Moreover, the weighted multiple averages under consideration can be applied to a broader  range of situations. It should be emphasized that the concepts of Diophantine rotations and analyticity,  employed in both quasi-periodic and almost periodic cases in this paper, stem from finite and infinite-dimensional KAM theory.  Specifically, systems that are nearly integrable, and possibly even nearly non-integrable, can be conjugated to simpler systems under certain assumptions, implying that the motion of angular variables corresponds to nonresonant rotations.  In cases where the original systems are analytic (refer to Definitions \ref{Finite-dimensional analyticity} and \ref{Infinite-dimensional analyticity}) and the rotations are of Diophantine types (refer to Definitions \ref{Finite-dimensional Diophantine condition} and \ref{wuqiongdio}), the resulting conjugations are also analytic, in both finite and infinite-dimensional contexts.  This corresponds to our Theorems \ref{MCOCO4.5} and \ref{MCOCO4.4} (or \ref{MCOCO4.6}), where the former guarantees exponential convergence, while the latter allows for arbitrary polynomial (or weaker exponential) convergence.
 
  When the rotational vectors exhibit weakly nonresonant behavior, such as the Bruno condition in the finite-dimensional case, the conjugations can also possess analyticity. This concept can be extended to the infinite-dimensional case, highlighting the significance of the joint nonresonant conditions proposed in Definitions \ref{Fi} and \ref{In}. It is well known, since Moser, that when the original systems in finite dimensions are of finite differentiability (with a suitably high order), the conjugations will admit finite smoothness. This corresponds to our Theorem \ref{MCORO1}, which shows polynomial convergence. Additionally, some mathematicians have further investigated Gevrey regularity instead of analyticity in KAM theory, and one could also establish corresponding Gevrey corollaries based on Theorems \ref{MT1} to \ref{MT4}. However, we will not delve into that aspect here.

\begin{theorem}[Polynomial convergence in the finite-dimensional case under finite smoothness] \label{MCORO1}
	Let $ \mathscr{B}=\mathbb{R}^p $ (equipped with the sup norm $ |\cdot| $) and each $ F_j $ be a $ C^{M_j} $ smooth map from $ \mathbb{T}^d $ to $ \mathscr{B} $, where $ p,d\in\mathbb{N}^+ $, $ 2 \leqslant M_j \leqslant +\infty $ and $ 1 \leqslant j \leqslant \ell $. Assume that there exists $ 2 \leqslant m \in \mathbb{N}^+ $ such that  the rotational vectors $ \{\rho_j\}_{j=1}^{\ell} $ satisfy the Finite-dimensional joint Diophantine condition in Definition \ref{Finite-dimensional Diophantine condition} with $ \mathop {\min }\nolimits_{1 \leqslant j \leqslant \ell } {M_j}>d+m\tau $. Then the weighted multiple ergodic averages $ \mathrm{DMW}_N^\ell $ and $ \mathrm{CMW}_T^\ell $ exhibit polynomial convergence of degree $ m $, i.e., \eqref{MT1lisan} and \eqref{MT1lianxu} with  $ 2 \leqslant m \in \mathbb{N}^+ $, respectively.
\end{theorem}
\begin{remark}
	In particular, if $ M_j=+\infty $ for all $ 1 \leqslant j \leqslant \ell $, then the resulting polynomial convergence rate can be  arbitrary, as discussed  in  Section \ref{Arbitrary polynomial convergence3}.
\end{remark}
\begin{proof}
	Let us verify the boundedness condition \eqref{jointcon1}. 
	In this case, we have $ \Delta(x)=x^\tau $ with some $ \tau=\tau(d,\ell) $ and $ \tilde\Delta_j(x)=x^{M_j} $ with $ 1 \leqslant j \leqslant \ell $ due to integration by parts. Taking into account Lemma \ref{tao} and the condition $ {M_j} - m\tau  - d + 1 > 1 $ for all $ 1 \leqslant j \leqslant \ell $, we deduce that
	\begin{align*}
		\sum\limits_{0 \ne {k^1}, \ldots ,{k^\ell } \in {\mathbb{Z}^d}} {\frac{{{\Delta ^m}\left( {\sum\nolimits_{j = 1}^\ell  {\left\| {{k^j}} \right\|} } \right)}}{{\prod\nolimits_{j = 1}^\ell  {{{\tilde \Delta }_j}\left( {\left\| {{k^j}} \right\|} \right)} }}}  &= \sum\limits_{0 \ne {k^1}, \ldots ,{k^\ell } \in {\mathbb{Z}^d}} {\frac{{{{\left( {\sum\nolimits_{j = 1}^\ell  {\left\| {{k^j}} \right\|} } \right)}^{m\tau }}}}{{\prod\nolimits_{j = 1}^\ell  {{{\left\| {{k^j}} \right\|}^{{M_j}}}} }}}  \\
		&\leqslant C\left( {m\tau ,\ell } \right)\sum\limits_{0 \ne {k^1}, \ldots ,{k^\ell } \in {\mathbb{Z}^d}} {\frac{{\sum\nolimits_{j = 1}^\ell  {{{\left\| {{k^j}} \right\|}^{m\tau }}} }}{{\prod\nolimits_{j = 1}^\ell  {{{\left\| {{k^j}} \right\|}^{{M_j}}}} }}} \\
		& =\mathcal{O}\left( {\int_1^{ + \infty } { \cdots \int_1^{ + \infty } {\frac{{\sum\nolimits_{j = 1}^\ell  {r_j^{m\tau }} }}{{\prod\nolimits_{j = 1}^\ell  {r_j^{{M_j} - d + 1}} }}d{r_1} \cdots d{r_\ell }} } } \right) \\
		&= \mathcal{O}\left( 1 \right).
	\end{align*}
 Therefore, Theorem \ref{MCORO1} is proved by directly applying Theorem \ref{MT1}.
\end{proof}

\begin{theorem}[Arbitrary polynomial convergence in the  infinite-dimensional case]\label{MCOCO4.5}
	Assume that the rotational vectors $ \{\rho_j\}_{j=1}^{\ell} $ satisfy the  Infinite-dimensional joint Diophantine condition in Definition \ref{wuqiongdio}, and $ \{F_j\}_{j=1}^{\ell} $ are analytic in $ \mathcal{G}(\mathbb{T}_{\sigma}^{\infty}) $ in Definition \ref{Infinite-dimensional analyticity}. Then the weighted multiple ergodic averages $ \mathrm{DMW}_N^\ell $ and $ \mathrm{CMW}_T^\ell $ exhibit  arbitrary polynomial convergence, i.e.,  $ \mathcal{O}\left(N^{-m}\right) $ and $ \mathcal{O}\left(T^{-m}\right) $ in Theorem \ref{MT2} with any fixed $ m \in \mathbb{N}^+ $, respectively.
\end{theorem}
\begin{remark}
We emphasize that $ C^\infty $ is sufficient to ensure arbitrary polynomial convergence for almost all rotations over $ \mathbb{T}^\infty $. The key point is to introduce a more general spatial structure, which we shall omit here for the sake of brevity. As it can be seen later, analyticity indeed leads to exponential convergence, through more accurate estimates.
\end{remark}
\begin{proof}
It suffices to verify the boundedness condition \eqref{jointcon2}. Recalling Lemma \ref{5.2}, we have $ {\vartheta}\left( x \right) = \mathcal{O}\left( {{e^{{\rho _ * }x}}} \right) $ with any $ \rho_*>0 $ under the joint Diophantine nonresonance in Definition \ref{wuqiongdio}. On these grounds, for any given $ 2 \leqslant m \in \mathbb{N}^+ $, fix $ {\rho _ * } = m^{-1}\left( {2\pi  - 2} \right)\sigma>0  $. Note that with $ \eta \geqslant 2 $ we have
	\[{\nu ^{{\nu ^{1/\eta }}}} = \exp \left( {{\nu ^{1/\eta }}\log \nu } \right) = \mathcal{O}\left( {\exp \left( {\sigma \nu } \right)} \right).\]
 Then it follows from Lemma \ref{wuqiongweijishu} that
	\begin{align}
		\sum\limits_{0 \ne k \in \mathbb{Z}_ * ^\infty } {\frac{1}{{\exp \left( {2\sigma {{\left| k \right|}_\eta }} \right)}}}  &= \mathcal{O}\left( {\sum\limits_{\nu  = 1}^\infty  {\sum\limits_{0 \ne k \in \mathbb{Z}_ * ^\infty ,{{\left| k \right|}_\eta } = \nu } {\frac{1}{{\exp \left( {2\sigma {{\left| k \right|}_\eta }} \right)}}} } } \right)\notag \\
		& = \mathcal{O}\left( {\sum\limits_{\nu  = 1}^\infty  {\left( {\sum\limits_{0 \ne k \in \mathbb{Z}_ * ^\infty ,{{\left| k \right|}_\eta } = \nu } 1 } \right) \cdot \frac{1}{{\exp \left( {2\sigma \nu } \right)}}} } \right)\notag \\
		& = \mathcal{O}\left( {\sum\limits_{\nu  = 1}^\infty  {\frac{{{\nu ^{{\nu ^{1/\eta }}}}}}{{\exp \left( {2\sigma \nu } \right)}}} } \right)\notag \\
	\label{expsigma}	& = \mathcal{O}\left( {\sum\limits_{\nu  = 1}^\infty  {\frac{1}{{\exp \left( {\sigma \nu } \right)}}} } \right) = \mathcal{O}\left( 1 \right).
	\end{align}
	Note that $ {{\tilde \Delta }_{\infty j}}\left( x \right) = {\mathcal{O}^\# }\left( {{e^{2\pi \sigma x}}} \right) $ for all $ 1 \leqslant j \leqslant \ell $ due to the analyticity of $ \{F_j\}_{j=1}^{\ell} $ in $ \mathcal{G}(\mathbb{T}_\sigma^\infty) $ with $ \sigma>0 $, see Definition \ref{Infinite-dimensional analyticity}. Therefore, by using  \eqref{expsigma}, we arrive at
	\begin{align*}
		\sum\limits_{0 \ne {k^1}, \ldots ,{k^\ell } \in \mathbb{Z}_ * ^\infty } {\frac{{{{\vartheta}^m}\left( {\sum\nolimits_{j = 1}^\ell  {{{\left| {{k^j}} \right|}_\eta }} } \right)}}{{\prod\nolimits_{j = 1}^\ell  {{{\tilde \Delta }_{\infty j}}\left( {{{\left| {{k^j}} \right|}_\eta }} \right)} }}}  &= \mathcal{O}\left( {\sum\limits_{0 \ne {k^1}, \ldots ,{k^\ell } \in \mathbb{Z}_ * ^\infty } {\frac{{\exp \left( {{m\rho _ * }\sum\nolimits_{j = 1}^\ell  {{{\left| {{k^j}} \right|}_\eta }} } \right)}}{{\prod\nolimits_{j = 1}^\ell  {\exp \left( {2\pi \sigma {{\left| {{k^j}} \right|}_\eta }} \right)} }}} } \right)\\
		& = \mathcal{O}\left( {\sum\limits_{0 \ne {k^1}, \ldots ,{k^\ell } \in \mathbb{Z}_ * ^\infty } {\frac{1}{{\prod\nolimits_{j = 1}^\ell  {\exp \left( {2\sigma \sum\nolimits_{j = 1}^\ell  {{{\left| {{k^j}} \right|}_\eta }} } \right)} }}} } \right)\\
		& = \mathcal{O}\left( {\sum\limits_{0 \ne k \in \mathbb{Z}_ * ^\infty } {\frac{1}{{\exp \left( {2\sigma {{\left| k \right|}_\eta }} \right)}}} } \right)= \mathcal{O}\left( 1 \right),
	\end{align*}
	which verifies \eqref{jointcon2}. Then  Theorem \ref{MCOCO4.5} is proved by applying Theorem \ref{MT2} since $ 2 \leqslant m \leqslant \mathbb{N}^+  $ could be arbitrarily fixed.
\end{proof}

\begin{theorem}[Universal exponential convergence in the  finite-dimensional case]\label{MCOCO4.4}
Assume $ \{F_j\}_{j=1}^{\ell} $ are analytic in $ \mathcal{G}(\mathbb{T}_{\sigma}^{d}) $ in Definition \ref{Finite-dimensional analyticity}. Then the weighted multiple ergodic averages $ \mathrm{DMW}_N^\ell $ and $ \mathrm{CMW}_T^\ell $ exhibit  exponential convergence, i.e.,  $ \mathcal{O}\left(\exp ( { - {N^{{\zeta_1}}}} )\right) $ and $ \mathcal{O}\left(\exp ( { - {T^{{\zeta_1}}}} )\right) $ in Theorem \ref{MT3} with some $ \zeta_1>0  $, respectively.
\end{theorem}
\begin{remark}
	In fact, it can be verified that the requirement of analyticity can be weakened to Gevrey regularity. Recalling Remark \ref{Remarkprobability}, we can  conclude that exponential convergence in the finite-dimensional case is universal when assuming analyticity and using our accelerated weighting method.
\end{remark}
\begin{proof}
	Let us consider the  Finite-dimensional joint Diophantine nonresonance in Definition \ref{Finite-dimensional Diophantine condition} for almost all  vectors $ \{\rho_j\}_{j=1}^{\ell} $ over $ \mathbb{T}^d $.	It should be pointed out that the truncated smallness  condition \eqref{T3jointcon1} might not hold, we therefore present the analysis of Theorem \ref{MCOCO4.4} based on the proof  of Theorem \ref{MT3}. Note that $ \Delta \left( x \right) = {x^\tau } $ with $ \tau >d \ell -1 $, and $ {{\tilde \Delta }_j}\left( x \right) = {\mathcal{O}^\# }\left( {{e^{2\pi \sigma x}}} \right) $ with $ \sigma >0 $ for all $ 1 \leqslant j \leqslant \ell $ due to the analyticity of $ \{F_j\}_{j=1}^{\ell} $ in $ \mathcal{G}(\mathbb{T}_\sigma^d) $. For the absolute constant $ \beta_1>0 $ in Theorem \ref{MT3}, let $ \varepsilon  = {\left( {\tau {\beta _1} + 1} \right)^{ - 1}} \in \left( {0,1} \right) $ and choose the adaptive function as $ \varphi \left( x \right) = {x^\varepsilon } $. Therefore, for $ 0<{\zeta _1} < {\beta _1}\varepsilon  = \left( {1 - \lambda } \right){\tau ^{ - 1}} $, we have
	\begin{align}
		\sum\limits_{0 \ne \sum\nolimits_{j = 1}^\ell  {{k^j}}  \in {\mathbb{Z}^{d\ell }}\backslash \mathscr{S}\left( x \right)} {\frac{1}{{\prod\nolimits_{j = 1}^\ell  {{{\tilde \Delta }_j}\left( {\left\| {{k^j}} \right\|} \right)} }}}  &= \mathcal{O}\left( {\sum\limits_{k \in {\mathbb{Z}^{d }},\left\| k \right\| > {\ell ^{ - 1}}{\Delta ^{ - 1}}\left( {x/\varphi \left( x \right)} \right)} {\frac{1}{{{{\tilde \Delta }_j}\left( {\left\| k \right\|} \right)}}} } \right)\notag \\
		& = \mathcal{O}\left( {\sum\limits_{k \in {\mathbb{Z}^d},\left\| k \right\| > {\ell ^{ - 1}}{x^{\left( {1 - \lambda } \right){\tau ^{ - 1}}}}} {\frac{1}{{\exp \left( {2\pi \sigma \left\| k \right\|} \right)}}} } \right)\notag \\
		&= \mathcal{O}\left( {\int_{{\ell ^{ - 1}}{x^{\left( {1 - \lambda } \right){\tau ^{ - 1}}}}}^{ + \infty } {\frac{{{r^{d - 1}}}}{{\exp \left( {2\pi \sigma r} \right)}}dr} } \right)\notag \\
		& =\mathcal{O}\left( {\int_{{\ell ^{ - 1}}{x^{\left( {1 - \lambda } \right){\tau ^{ - 1}}}}}^{ + \infty } {\frac{1}{{\exp \left( {\sigma r} \right)}}dr} } \right)\notag \\
		&= \mathcal{O}\left( {\exp \left( { - \sigma {\ell ^{ - 1}}{x^{\left( {1 - \lambda } \right){\tau ^{ - 1}}}}} \right)} \right)\notag \\
		\label{modifyzhishu1}	& = \mathcal{O}\left( {\exp \left( { - {x^{{\zeta _1}}}} \right)} \right).
	\end{align}
	Combining \eqref{modifyzhishu1} and \eqref{J1-2} (see the proof of Theorem \ref{MT3}) we finally have
	\begin{equation}\notag
		{\left\| {\mathrm{DMW}_N^\ell\left( \mathcal{F} \right)\left( \theta  \right) - \prod\limits_{j = 1}^\ell  \left({\int_{{\mathbb{T}^d}} {{F_j}(\hat \theta )d\hat \theta } }\right) } \right\|_\mathscr{B}} =  \mathcal{O}\left( {\exp \left( { - {N^{{\zeta _1}}}} \right)} \right)
	\end{equation}
	due to \eqref{J1+J2}. The continuous case follows the same approach as discussed above. This proves  Theorem \ref{MCOCO4.4}.
\end{proof}

\begin{theorem}[Exponential convergence in the  infinite-dimensional case]\label{MCOCO4.6}
	Assume that the rotational vectors $ \{\rho_j\}_{j=1}^{\ell} $ satisfy the  Infinite-dimensional joint Diophantine condition in Definition \ref{In} with $ 2 \leqslant \mu=\eta\in \mathbb{N}^+ $, and $ F_j \in \mathscr{B}_{\tilde\Delta_\infty} $ with $ \tilde\Delta_\infty(x) = \exp \left( {\exp  x } \right)  $ for all $ 1\leqslant j \leqslant \ell $. Then the weighted multiple ergodic averages $ \mathrm{DMW}_N^\ell $ and $ \mathrm{CMW}_T^\ell $ exhibit  exponential convergence, i.e.,  $ \mathcal{O}\left(\exp ( { - {N^{{\zeta_2}}}} )\right) $ and $ \mathcal{O}\left(\exp ( { - {T^{{\zeta_2}}}} )\right) $ in Theorem \ref{MT4} with some $ \zeta_2>0  $, respectively.
\end{theorem}
\begin{proof}
	It suffices to verify the truncated smallness  condition \eqref{T4jointcon1}. Note that Lemma \ref{5.2}\footnote{It appears that the exponent of the logarithmic term in the original estimate contained a minor inaccuracy; however, this in no way affects any of the results obtained when the estimate is applied in this paper.} implies $ {{\vartheta}(x)}=\mathcal{O}(e^{x/2\ell}) $, which leads to $ {{{\vartheta}^{-1}}}(x)\geqslant 2 \ell \log x$ for $ x>0 $ sufficiently large. Choose the adaptive function as $ \varphi \left( x \right) = \sqrt{x} $. By employing Lemma \ref{wuqiongweijishu}, we have
	\begin{align*}
		\sum\limits_{0 \ne \sum\nolimits_{j = 1}^\ell  {{k^j}}  \in \mathbb{Z}_ * ^\infty \backslash {\mathscr{S}_\infty }\left( x \right)} {\frac{1}{{\prod\nolimits_{j = 1}^\ell  {{{\tilde \Delta }_{\infty j}}\left( {{{\left| {{k^j}} \right|}_\eta }} \right)} }}}	&= \mathcal{O}\left( {\sum\limits_{k \in \mathbb{Z}_ * ^\infty ,{{\left| k \right|}_\eta } > {\ell^{-1}{\vartheta}^{ - 1}}\left( {x/\varphi \left( x \right)} \right)} {\frac{1}{{{{\tilde \Delta }_\infty }\left( {{{\left| k \right|}_\eta }} \right)}}} } \right)\\
		& = \mathcal{O}\left( {\sum\limits_{k \in \mathbb{Z}_ * ^\infty ,{{\left| k \right|}_\eta } > \log x} {\frac{1}{{\exp \left( {\exp \left( {{{\left| k \right|}_\eta }} \right)} \right)}}} } \right)\\
		& = O\left( {\sum\limits_{\nu  = \left[ {\log x} \right]}^\infty  {\sum\limits_{k \in \mathbb{Z}_ * ^\infty ,{{\left| k \right|}_\eta } = \nu } {\frac{1}{{\exp \left( {\exp \left( {{{\left| k \right|}_\eta }} \right)} \right)}}} } } \right)\\
		& = \mathcal{O}\left( {\sum\limits_{\nu  = \left[ {\log x} \right]}^\infty  {\frac{{{\nu ^{{\nu ^{1/\eta }}}}}}{{\exp \left( {\exp  \nu  } \right)}}} } \right)\\
		& = \mathcal{O}\left( {\sum\limits_{\nu  = \left[ {\log x} \right]}^\infty  {\frac{1}{{\exp \left( {\left(\exp  \nu\right)  /2} \right)}}} } \right) = \mathcal{O}\left( {{e^{ - x/3}}} \right),
	\end{align*}
	which verifies \eqref{T4jointcon1}. Then from Theorem \ref{MT4}, there exists $ \zeta_2>0 $ such that Theorem \ref{MCOCO4.6} holds.
\end{proof}

\begin{theorem}[Universal exponential convergence in the  infinite-dimensional case]\label{MCOCO4.7}
	Assume that  $ \{F_j\}_{j=1}^{\ell} $ are analytic in $ \mathcal{G}(\mathbb{T}_{\sigma}^{\infty}) $ in Definition \ref{Infinite-dimensional analyticity}. Then for almost all vectors $ \{\rho_j\}_{j=1}^{\ell} \in \mathbb{T}^\infty $,   the weighted multiple ergodic averages $ \mathrm{DMW}_N^\ell $ and $ \mathrm{CMW}_T^\ell $ exhibit exponential convergence, i.e.,
	\begin{equation}\label{indeedexp}
		\mathcal{O}\left(\exp \left( { - {({\log N})^{{\zeta_3}}}} \right)\right)\;\;  \text{and} \;\; \mathcal{O}\left(\exp \left( { - {({\log T})^{{\zeta_3}}}} \right)\right)
	\end{equation}
	in Theorem \ref{MT4} with some $ \zeta_3>1  $, respectively.
\end{theorem}
\begin{remark}
	Note that the convergence rate in \eqref{indeedexp} is indeed an exponential type due to $ \zeta_3>1 $, i.e., faster than an arbitrary polynomial type $ x^{-m} $ ($ m \in \mathbb{N}^+ $) since $ m\log x =o\left((\log x)^{\zeta_3} \right)   $ as $ x \to +\infty $. This demonstrates  that  exponential convergence is also universal in the infinite-dimensional case under analyticity and our accelerated weighting method, although it is slower than that in the finite-dimensional case, as shown in Theorem \ref{MCOCO4.4}. It is also somewhat surprising in the sense of eliminating the impact  of spatial dimension.
\end{remark}
\begin{proof}
	There are at least two approaches to prove the desired conclusion, namely by constructing a new universal nonresonant condition over $ \mathbb{T}^\infty $, or by establishing more accurate estimates of the infinite-dimensional Diophantine condition with  full probability measure. Actually, there are some interesting connections. Let us first discuss the former.

	Similar to that in Theorem \ref{MCOCO4.4}, the truncated smallness  condition \eqref{T4jointcon1} seems not hold via the infinite-dimensional analyticity, and we shall prove Theorem \ref{MCOCO4.7} based on the analysis of Theorem \ref{MT4}. The key point is to construct a special nonresonant condition for which almost all rotations on the infinite-dimensional torus $ \mathbb{T}^\infty $ hold, such that the exponential convergence rate in \eqref{indeedexp} could be achieved, under the analyticity regularity in $ \mathcal{G}(\mathbb{T}_{\sigma}^{\infty}) $ for $ \{F_j\}_{j=1}^{\ell} $, i.e., $ \tilde\Delta_{\infty j}(x)={\mathcal{O}}^{\#}\left(e^{2\pi \sigma x}\right) $ with $ \sigma>0 $ for all $ 1 \leqslant j \leqslant \ell $. 	Define the new approximation function $ {\vartheta}^*(x) $ for the infinite-dimensional continuous case in Definition \ref{In} as $ {\vartheta}^*(x)={\mathcal{O}}^{\#}\left(\exp(x^{3/4})\right) $, i.e., stronger than the normal exponential  type $ {\vartheta}_*(x)=e^x $ in the sense of dealing with small divisors, and let $ \rho \in \mathbb{T}^\infty $ satisfy the following nonresonant condition with some $ \gamma>0 $ (the discrete case is exactly the same):
	\begin{equation}\label{mance1}
		\left| {k \cdot \rho } \right| > \frac{\gamma }{{{{\vartheta}^ * }({{\left| k \right|}_\eta })}},\;\;\forall 0 \ne k \in \mathbb{Z}_ * ^\infty .
	\end{equation}
	Then with Lemma \ref{wuqiongweijishu} we get
	\begin{align*}
	  \sum\limits_{0 \ne k \in \mathbb{Z}_ * ^\infty } {\frac{\gamma}{{{{\vartheta}^ * }({{\left| k \right|}_\eta })}}}  &= \gamma \sum\limits_{0 \ne k \in \mathbb{Z}_ * ^\infty } {\frac{1}{{\exp \left( {\left| k \right|_\eta ^{3/4 }} \right)}}}  = \gamma \sum\limits_{\nu  = 1}^\infty  {\left( {\sum\limits_{0 \ne k \in \mathbb{Z}_ * ^\infty ,{{\left| k \right|}_\eta } = \nu } {\frac{1}{{\exp \left( {\left| k \right|_\eta ^{3/4 }} \right)}}} } \right)}  \\
		&= \gamma \cdot \mathcal{O}\left( {\sum\limits_{\nu  = 1}^\infty  {\frac{{{\nu ^{{\nu ^{1/\eta }}}}}}{{\exp \left( {{\nu ^{3/4 }}} \right)}}} } \right) = \gamma \cdot \mathcal{O}\left( {\sum\limits_{\nu  = 1}^\infty  {\frac{1}{{\exp \left( {{\nu ^{\varepsilon^* }}} \right)}}} } \right) =\gamma \cdot \mathcal{O}\left( 1  \right),
	\end{align*}
	where $ 0<\varepsilon^*<3/4-1/\eta $ due to $ 2 \leqslant \eta \in \mathbb{N}^+ $. This shows that almost all rotations on $ \mathbb{T}^\infty $ satisfy our new nonresonant condition \eqref{mance1}, i.e., form a set of full probability measure, because $ \gamma>0 $ could be arbitrarily small. A similar conclusion holds for joint nonresonance.	In this case, we have $ {{\vartheta}^ * }^{ - 1}\left( x \right) \geqslant {\left( {\log x} \right)^{4 /3}} $ for $ x>0 $ sufficiently large. Let us choose the adaptive function as $ \varphi \left( x \right) = \sqrt{x} $. Therefore, by using Lemma \ref{wuqiongweijishu} one derives that
	\begin{align*}
		&\;\;\;\;\sum\limits_{0 \ne \sum\nolimits_{j = 1}^\ell  {{k^j}}  \in \mathbb{Z}_ * ^\infty \backslash {\mathscr{S}_\infty }\left( x \right)} {\frac{1}{{\prod\nolimits_{j = 1}^\ell  {{{\tilde \Delta }_{\infty j}}\left( {{{\left| {{k^j}} \right|}_\eta }} \right)} }}}  \\
		&= \mathcal{O}\left( {\sum\limits_{k \in \mathbb{Z}_ * ^\infty ,{{\left| k \right|}_\eta } > {\ell}^{-1}{{\vartheta}^ * }^{ - 1}\left( {x/\varphi \left( x \right)} \right)} {\frac{1}{{{{\tilde \Delta }_\infty }\left( {{{\left| k \right|}_\eta }} \right)}}} } \right)\\
		& = \mathcal{O}\left( {\sum\limits_{k \in \mathbb{Z}_ * ^\infty ,{{\left| k \right|}_\eta } > {{\ell}^{-1}2^{-4/3}(\log x)^{4/3}}} {\frac{1}{{\exp \left( {2\pi \sigma {{\left| k \right|}_\eta }} \right)}}} } \right)\\
		& = \mathcal{O}\left( {\sum\limits_{\nu  = \left[ {{\ell}^{-1}2^{-4/3}(\log x)^{4/3}} \right]}^\infty  {\sum\limits_{k \in \mathbb{Z}_ * ^\infty ,{{\left| k \right|}_\eta } = \nu } {\frac{1}{{\exp \left( {2\pi \sigma {{\left| k \right|}_\eta }} \right)}}} } } \right)\\
		& = \mathcal{O}\left( {\sum\limits_{\nu  = \left[ {{\ell}^{-1}2^{-4/3}(\log x)^{4/3}} \right]}^\infty  {\frac{{{\nu ^{{\nu ^{1/\eta }}}}}}{{\exp \left( {2\pi \sigma \nu } \right)}}} } \right)\\
		& = \mathcal{O}\left( {\sum\limits_{\nu  = \left[ {{\ell}^{-1}2^{-4/3}(\log x)^{4/3}} \right]}^\infty  {\frac{1}{{\exp \left( {\sigma \nu } \right)}}} } \right)\\
		& = \mathcal{O}\left( {\exp \left( { - {{\left( {\log x} \right)}^{{\zeta _3}}}} \right)} \right),
	\end{align*}
	provided a universal constant $ 1 < {\zeta _3} < 4 /3 $. This gives the tail estimate for the weighted multiple ergodic average $ \mathrm{DMW}_N^\ell $.  Finally, recalling the similar arguments in the proof of Theorem \ref{MCOCO4.4} and observing that the continuous case is exactly the same, we prove the conclusion in Theorem \ref{MCOCO4.7}.

	Another feasible approach is  improving the Diophantine uniform estimates provided in Lemma \ref{5.2}, in the sense of truncation. The ``uniform'' means that $\prod\nolimits_{j \in \mathbb{N}} {\left( {1 + {{\left| {{k_j}} \right|}^\mu }{{\left\langle j \right\rangle }^\mu }} \right)}  $ could be  dominated by a function with variable $ {\left| k \right|_\eta } $, namely $ {e^{{\rho _ * }{{\left| k \right|}_\eta }}} $ in this case, and this is  consistent with the form of our nonresonant condition, see Definition \ref{In}. However, as shown by our  strategy in Section \ref{MULSEC6.3} (the almost periodic case is indeed  similar), we only need small divisor estimates in the truncated part $ \mathcal{J}_1 $, and taking sup of certain estimate in the range $ 0 \ne \sum\nolimits_{j = 1}^\ell  {{k^j}}  \in \mathscr{S}\left( N \right) $ is enough, see \eqref{zhishuzhishu} for details. As a consequence, the uniform estimate in Lemma \ref{5.2} is somewhat superfluous, and we shall establish a weaker estimate for ${\sup _{0 < {{\left| k \right|}_\eta } \leqslant N}}\prod\nolimits_{j \in \mathbb{N}} {\left( {1 + {{\left| {{k_j}} \right|}^\mu }{{\left\langle j \right\rangle }^\mu }} \right)}   $. For the sake of brevity, let us consider the Diophantine nonresonance in Definition \ref{wuqiongdio}   with $ 2 \leqslant \mu=\eta\in \mathbb{N}^+ $. For fixed $ k \in \mathbb{Z}_ * ^\infty  $, denote by $ m $ the number of nonzero components of $ k $. Then $ {\left| k \right|_\eta } \leqslant N $ will lead to 
	\begin{align}
		N &\geqslant {\left| k \right|_\eta } = \sum\limits_{j \in \mathbb{N}} {\left| {{k_j}} \right|{{\left\langle j \right\rangle }^\eta }}  = \sum\limits_{i = 1}^m {\left| {{k_{{j_i}}}} \right|{{\left\langle {{j_i}} \right\rangle }^\eta }}  \geqslant \sum\limits_{i = 1}^m {{{\left\langle {{j_i}} \right\rangle }^\eta }} \notag \\
	\label{MULn}	& \geqslant \sum\limits_{i = 1}^m {{i^\eta }}  = {\mathcal{O}^\# }\left( {\int_1^m {{x^\eta }dx} } \right) = {\mathcal{O}^\# }\left( {{m^{1 + \eta }}} \right),
	\end{align}
	i.e., $ m = \mathcal{O}^\#\left( {{N^{1/\left( {1 + \eta } \right)}}} \right) $. Now, with the observation
	\[{\left| {{k_j}} \right|^\eta }{\left\langle j \right\rangle ^\eta } \leqslant {\left( {\left| {{k_j}} \right|{{\left\langle j \right\rangle }^\eta }} \right)^\eta } \leqslant {\left( {\sum\limits_{j \in \mathbb{N}} {\left| {{k_j}} \right|{{\left\langle j \right\rangle }^\eta }} } \right)^\eta } = \left| k \right|_\eta ^\eta  \leqslant {N^\eta },\]
	one obtains the followings via a universal constant $ {{C_\eta }}>0 $ only depending on $ \eta $:
	\begin{align}
\mathop {\sup }\limits_{0 < {{\left| k \right|}_\eta } \leqslant N} \prod\limits_{j \in \mathbb{N}} {\left( {1 + {{\left| {{k_j}} \right|}^\eta }{{\left\langle j \right\rangle }^\eta }} \right)}  &= \mathop {\sup }\limits_{0 < {{\left| k \right|}_\eta } \leqslant N} \exp \left( {\sum\limits_{j \in \mathbb{N}} {\ln \left( {1 + {{\left| {{k_j}} \right|}^\eta }{{\left\langle j \right\rangle }^\eta }} \right)} } \right)\notag \\
	& \leqslant \exp \left( {\sum\limits_{ i  = 1}^m {\ln \left( {1 + {N^\eta }} \right)} } \right)\notag \\
\label{MULXIAOCHUSHU}	& \leqslant \exp \left( {{C_\eta }{N^{1/\left( {1 + \eta } \right)}}\ln N} \right).
	\end{align}
 An interesting fact is that our estimate \eqref{MULXIAOCHUSHU} is indeed  optimal. Note that if one considers that the nonzero components of $ \tilde k \in \mathbb{Z}_ * ^\infty $ are all modulus $ 1 $, and they are consecutively from $ 0 $, e.g., $\tilde k = ( \ldots ,0,\underbrace {1,1, \ldots ,1}_m,0, \ldots ) $,  then \eqref{MULn} tells  us that $ m = \mathcal{O}^\#\left( {{N^{1/\left( {1 + \eta } \right)}}} \right)   $ is also valid. Now we get
\begin{align*}
	\prod\limits_{j \in \mathbb{N}} {\left( {1 + {{| {{\tilde k_j}} |}^\eta }{{\left\langle j \right\rangle }^\eta }} \right)}  &= \exp \left( {\sum\limits_{j = 0}^m {\ln \left( {1 + {{\left\langle j \right\rangle }^\eta }} \right)} } \right) = \exp \left( {\eta \mathcal{O}^\#\left( {\int_1^m {\ln xdx} } \right)} \right)\\
	& = \exp \left( {\eta \mathcal{O}^\#\left( {m\ln m} \right)} \right) = \exp \left( {\mathcal{O}^\#\left( {{N^{1/\left( {1 + \eta } \right)}}\ln N} \right)} \right),
\end{align*}
which is the same as \eqref{MULXIAOCHUSHU} and shows the promised optimality. Note that \eqref{MULXIAOCHUSHU} is non-uniform, because it is not the type $ \exp \left( {{C_\eta }\left| k \right|_\eta ^{1/\left( {1 + \eta } \right)}\ln {{\left| k \right|}_\eta }} \right) $. However, estimate \eqref{MULXIAOCHUSHU} is enough to allow us achieving exponential convergence, and it is essentially  the same (including the rate) as the first approach,  except here we are using the universal Diophantine condition (and therefore we do not  need additional proof of universality, but one notices that the proof of universality could actually be similar). 
\end{proof}

\section{Optimality of the convergence rate and inremovability of the nonresonant jointness}\label{Optimality of convergence rate}
One observes that the weighting function \eqref{expfun} could indeed be replaced by any $ C^{\infty} $ function $ {\tilde w}>0 $, satisfying
\begin{equation}\notag
	{{\tilde w}^{(j)}}(0) = {{\tilde w}^{(j)}}(1),\;\; 0 \leqslant j \leqslant m-1
\end{equation}
and $ \int_0^1 {\tilde w\left( x \right)dx}  = 1 $ in Theorems \ref{MT1} and \ref{MT2}, and the  polynomial convergence  could be automatically preserved. If $ \{F_j\}_{j=1}^{\ell} $ are $ C^\infty $ with $ 1 \leqslant j \leqslant \ell $ and  the rotational vectors $  \{\rho_j\}_{j=1}^\ell  \in \mathbb{T}^d $ satisfy the Finite-dimensional joint Diophantine condition in Definition \ref{Finite-dimensional Diophantine condition}, then the corresponding weighted multiple ergodic averages with the weighting function $ \tilde{w} $ exhibit polynomial convergence of degree $ m $ at this time, i.e., $ \mathcal{O}(N^{-m}) $ and $ \mathcal{O}(T^{-m}) $ in Theorem \ref{MT1}, respectively. That is, the convergence rate also depends on the degree $ m $ of the weighting function $ \tilde{w} $, which is consistent with the conclusion in \cite{MR3755876}.

However, the requirement $ w\in C_0^\infty \left( {\left[ {0,1} \right]} \right) $ for our exponential  weighting function $ w $ throughout this paper cannot be removed. Otherwise, the resulting convergence rate would be \textit{at most} a finite order polynomial  type, i.e., the arbitrary polynomial convergence cannot be achieved, even if the boundedness conditions \eqref{jointcon1} and \eqref{jointcon2} are satisfied for  any $ 2 \leqslant m \in \mathbb{N}^+ $ and the observables are analytic. Actually, for trigonometric polynomials (which are, of course, analytic) where small divisors are absent and hence no infinite summation is involved,  it is evident to verify that the polynomial convergence rate can be improved to degree $ m+1 $, i.e., $ \mathcal{O}(N^{-m-1}) $ and $ \mathcal{O}(T^{-m-1}) $, due to the fact that
\[\left| {\int_0^1 {{w^{\left( m \right)}(y)}\exp \left( {iaNy} \right)dy} } \right| = \mathcal{O}\left( {{N^{ - 1}}} \right),\]
see details from \eqref{MUL6.6}. \textit{This is optimal in terms of the convergence rate, i.e., the degree $ m+1 $ of the polynomial convergence  cannot be replaced by any number less than it, as illustrated by the counterexample provided below. This also demonstrates a certain optimality of our weighting method in terms of preserving rapid convergence, as previously mentioned.}

Here we provide a counterexample in the true multiple setting: the $ 2 $-dimensional continuous  case, i.e., $ \ell=2 $. Namely, let the weighting function  $ {w_{{{\sin }^2}}}\left( x \right) = 2{\sin ^2}\left( {\pi x} \right) $ and $ F_j\left( x \right) = \sin \left( {2\pi x} \right) $ with $ j=1,2 $ in $ \left(\mathbb{R}^1,|\cdot|\right) $ be given. Then one easily checks that $ \int_{{\mathbb{T}^1}} {F_j(\hat \theta )d\hat \theta }  = 0 $, and $ \int_0^1 {{w_{{{\sin }^2}}}\left( x \right)dx}  = 1 $,  $ w_{{{\sin }^2}}^{\left( \nu \right)}\left( 0 \right) = w_{{{\sin }^2}}^{\left( \nu \right)}\left( 1 \right) = 0 $ with $ 0 \leqslant \nu \leqslant 1 $ and $ w_{{{\sin }^2}}^{\left( 2 \right)}\left( 0 \right) \ne 0 $, i.e., $ m=2 $. Assume that $ \rho_1\ne \rho_2  $ are irrational, and satisfy 
\begin{equation}\label{rho12}
	\upsilon : = \frac{{{\rho _1} + {\rho _2}}}{{{\rho _1} - {\rho _2}}} = \frac{4}{\pi } \notin \mathbb{Q}.
\end{equation}
Then the corresponding weighted multiple ergodic average in the continuous case with the weighting function $ {w_{{{\sin }^2}}}\left( x \right) $ and the initial point $ \theta=0 $ can be expressed as
\begin{align*}
	\mathcal{H}\left( T \right): &= \frac{1}{T}\int_0^T {{w_{{{\sin }^2}}}\left( {s/T} \right)F_1\left( {\mathscr{T}_{{\rho _1}}^s\left( \theta  \right)} \right)F_2\left( {\mathscr{T}_{{\rho _2}}^s\left( \theta  \right)} \right)ds} \\
	& = 2\int_0^1 {{{\sin }^2}\left( {\pi y} \right)\sin \left( {2\pi T\rho_1 y} \right)\sin \left( {2\pi T\rho_2 y} \right)dy}\\
	& = \frac{1}{{4\pi }}\left[ {\frac{{\sin \left( {2\pi \left( {{\rho _1} + {\rho _2}} \right)T} \right)}}{{\left( {{\rho _1} + {\rho _2}} \right)T\left( {{{\left( {{\rho _1} + {\rho _2}} \right)}^2}{T^2} - 1} \right)}} - \frac{{\sin \left( {2\pi \left( {{\rho _1} - {\rho _2}} \right)T} \right)}}{{\left( {{\rho _1} - {\rho _2}} \right)T\left( {{{\left( {{\rho _1} - {\rho _2}} \right)}^2}{T^2} - 1} \right)}}} \right].
\end{align*}
Taking $ {T_n}: = {\left( {{\rho _1} - {\rho _2}} \right)^{ - 1}}n =\mathcal{O}^{\#}(n)$, one verifies that $ \left( {{\rho _1} + {\rho _2}} \right){T_n} \notin {\mathbb{N}^ + } $ by \eqref{rho12}. Then we derive 
\[\left| {\mathcal{H}\left( {{T_n}} \right)} \right| = \frac{{\left| {\sin \left( {8n} \right)} \right|}}{{4\pi \upsilon n\left( {{\upsilon ^2}{n^2} - 1} \right)}}\]
by the choice of $ T_n $. Hence there exists a subsequence $ {\left\{ {{T_{{n_k}}}} \right\}_k} $ of $ {\left\{ {{T_{{n}}}} \right\}_n} $ such that
\begin{equation}\label{>3}
	\left| {\mathcal{H}\left( {{T_{{n_k}}}} \right)} \right| > {\xi}T_{{n_k}}^{ - 3}
\end{equation}
holds with a universal constant $ \xi >0$ independent of $ n_k $, due to the density of $ \{\left|\sin(8n)\right|\}_n $ on the interval $ [0,1] $.
This shows that the pointwise convergence of the weighted  multiple ergodic average is indeed a polynomial  type of degree $ 3 $ because \eqref{>3} and $ \left| {\mathcal{H}\left( T \right)} \right| = \mathcal{O}\left( {{T^{ - 3}}} \right) $ which we forego, and it is consistent with the numerical simulation results  presented in  Section 3 in \cite{MR3718733} (although in the $ 1 $-dimensional case, while we consider the true multiple ergodic average here).

By the way, our proposed joint  nonresonant conditions in Definitions \ref{Fi} and \ref{In} are necessary for our results in this paper. Namely, letting $ {\rho _1} = {\rho _2} = \rho  $ be irrational numbers,  we have
\begin{align*}
	\mathop {\lim }\limits_{T \to  + \infty } \mathcal{H}\left( T \right) &= \mathop {\lim }\limits_{T \to  + \infty } 2\int_0^1 {{{\sin }^2}\left( {\pi y} \right){{\sin }^2}\left( {2\pi T\rho y} \right)dy} \\
	& = \mathop {\lim }\limits_{T \to  + \infty } \frac{1}{8}\left( {4 - \frac{{\sin \left( {4\pi T\rho } \right)}}{{\pi T\rho  - 4\pi {T^3}{\rho ^3}}}} \right) \\
	&= \frac{1}{2} \ne 0=\prod\limits_{j = 1}^2  \left({\int_{{\mathbb{T}^1}} {{F_j}(\hat \theta )d\hat \theta } }\right),
\end{align*}
which shows the inremovability of the nonresonant jointness. This is indeed a distinction from the $ 1 $-dimensional case discussed by the  authors in  \cite{MR4768308}.

\section{Numerical simulation and analysis of convergence rates}\label{MULSEC6}
To better  clarify our results, we provide in this section some explicit examples of  weighted multiple Birkhoff averages and perform  numerical simulations on them.

\subsection{The comparison of the convergence rates}\label{MULSUBSEC1}
Here, our main focus is on the significant enhancement of the convergence rate through our weighting method. Therefore, we will construct a \textit{multiple} example below to compare the convergence rates of the unweighted type and the weighted type.
 
Let the multiple index be $ \ell  = 2 $, the $ 1 $-dimensional analytic observables be $ {F_1}\left( x \right) = {F_2}\left( x \right) = \sin \left( {2\pi x} \right) $,  the rotations be  $ {\rho _1} = ( {\sqrt 5  - 1} )/2$ and $ {\rho _2} = 1 $, and the initial point be $ \theta=10^{-1} $. Then the joint nonresonant rotation becomes $ \tilde \rho  = \left( {(\sqrt 5  - 1)/2,1} \right) $. Indeed, it admits the exact  Diophantine exponent of $ 1 $, because for some absolute number $ \gamma >0 $ and all $ k = \left( {{k_1},{k_2}} \right) \in {\mathbb{Z}^2} , m \in \mathbb{Z}$, it follows from the arithmetic property of the  Golden Number $ {\rho _1} = ( {\sqrt 5  - 1} )/2$ that (it is the constant type, see Herman \cite{MR538680} for instance) 
\[\left| {k \cdot \tilde \rho  - m} \right| = \left| {{k_1} \cdot (\sqrt 5  - 1)/2 - \left( {-{k_2} + m} \right)} \right| \geqslant \frac{\gamma }{{\left| {{k_1}} \right|}} \geqslant \frac{\gamma }{{\left| k \right|}},\]
and the exponent $ 1 $ is optimal when considering  $ {k_1} = {p_l} $ and $ -{k_2} + m = {q_l} $  with $ {\left\{ {{p_l}/{q_l}} \right\}_{l \in {\mathbb{N}^ + }}} $ being the approximants of $ \rho_1 $. Choosing such well-nonresonant vectors (i.e., those far from the rational ones) can reduce  errors in practical calculations, even though our theorems still guarantee theoretical exponential convergence for other, weaker  nonresonant ones.

\begin{figure}[h] 	\centering 	\includegraphics[width=350pt]{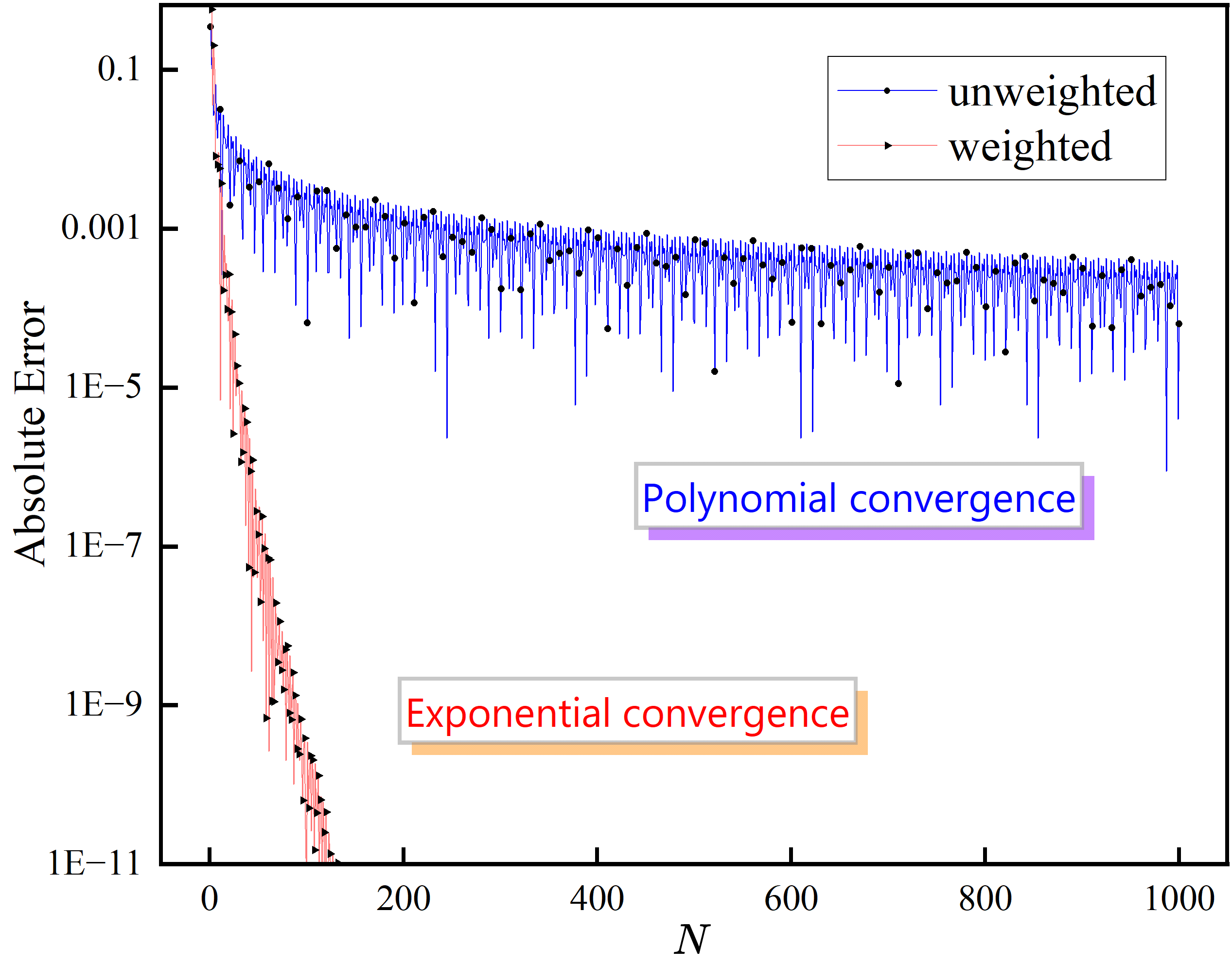} 	\caption {Comparison  of the  convergence rates} 	\label{FIGMUL} \end{figure}

 Under this setting, the unweighted multiple ergodic average (the multiple Birkhoff type) corresponds to
\[{\frac{1}{N}\sum\limits_{n = 0}^{N - 1} {\sin \left( {2\pi \left( {\frac{1}{{10}} + \frac{{\sqrt 5  - 1}}{2}n} \right)} \right)\sin \left( {2\pi \left( {\frac{1}{{10}} + n} \right)} \right)} },\]
and our weighted multiple ergodic average  $ {{\rm{DMW}}_N^\ell \left( \mathcal{F} \right)\left( \theta  \right)} $ is expressed as
\[{\frac{1}{{{A_N}}}\sum\limits_{n = 0}^{N - 1} {w\left( {\frac{n}{N}} \right)\sin \left( {2\pi \left( {\frac{1}{{10}} + \frac{{\sqrt 5  - 1}}{2}n} \right)} \right)\sin \left( {2\pi \left( {\frac{1}{{10}} + n} \right)} \right)} }.\]
As observed by Mondal et al. in a recent work \cite{arXiv:2404.11507}, oscillation around the mean is a universal behavior in such ergodic averages (at least for the unweighted type). In our case, both two averages oscillate around the mean $\prod\nolimits_{j = 1}^2 {\int_{{\mathbb{T}^1}} {{F_j}(\hat \theta )d\hat \theta } }  = {\left( {\int_0^1 {\sin (2\pi \hat \theta )d\hat \theta } } \right)^2}= 0 $.  Therefore,  we choose to calculate the \textit{absolute values} of these two averages in Figure \ref{FIGMUL}. As shown in Figure \ref{FIGMUL}, our weighting method can improve the slow polynomial convergence (bule) into a rapid exponential convergence (red). In fact, the  polynomial convergence rate of the unweighted type is $ \mathcal{O}(N^{-1}) $, either by co-boundary construction  or by direct calculation.

\subsection{The indispensability of the  balancing conditions}
In this subsection, we will demonstrate through numerical simulations that our balancing conditions are indispensable for achieving rapid convergence.
\begin{figure}[h] 	\centering 	\includegraphics[width=350pt]{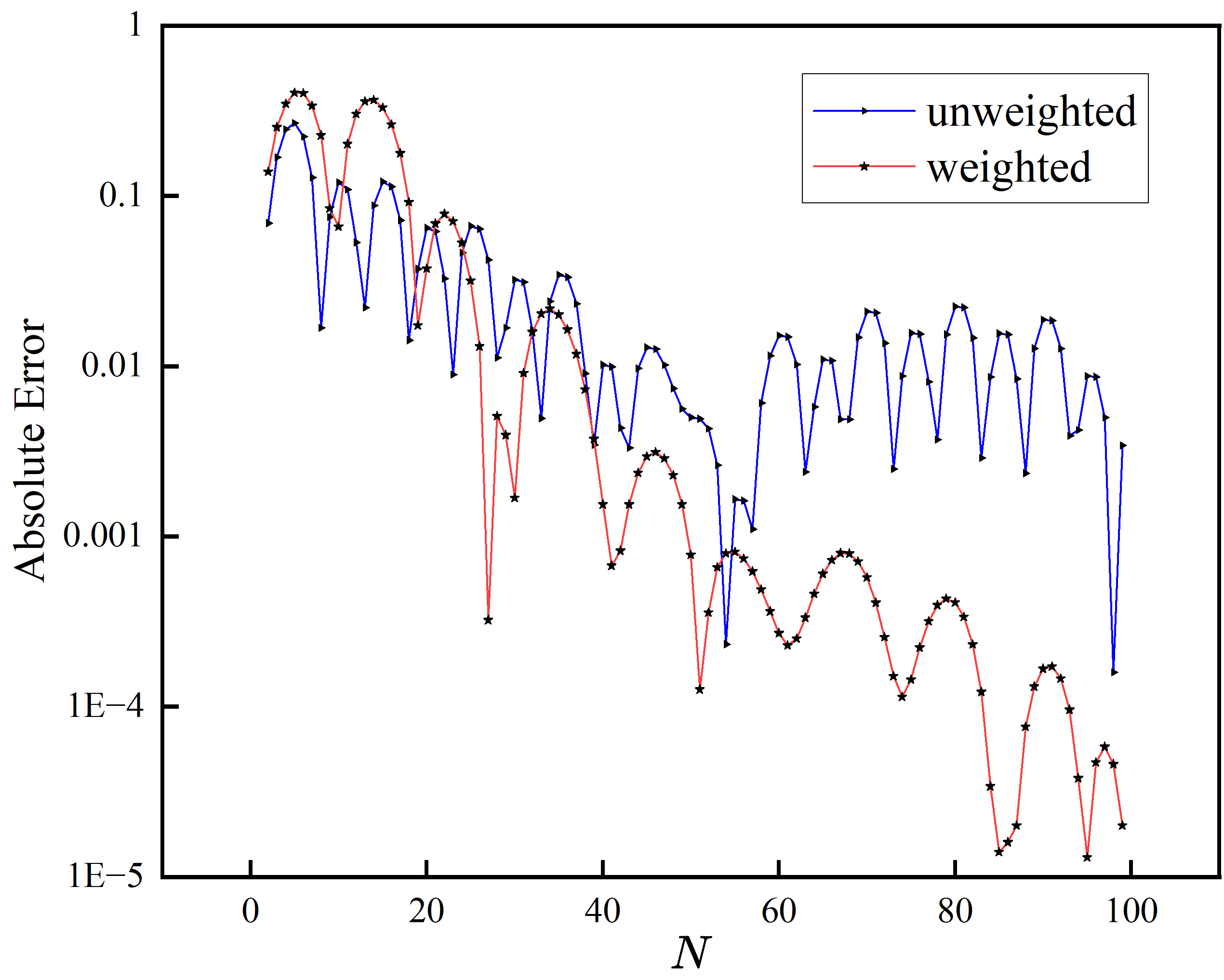} 	\caption {The impact of regularity and irrationality on the  convergence rates} 	\label{FIGMUL2} \end{figure}
Our balancing conditions, namely \eqref{jointcon1}, \eqref{jointcon2}, \eqref{T3jointcon1} and \eqref{T4jointcon1}, establish the effect of the regularity of the observables and the irrationality of the joint rotational vectors on the convergence rates. One may ask whether they are \textit{essential}.
Below we will construct a \textit{multiple} example with weak regularity and weak irrationality and numerically simulate the convergence rates in Figure \ref{FIGMUL2} to address  this question: if our balancing conditions are not satisfied, the convergence rate may not be very fast. 

Let the multiple index be $ \ell  = 2 $, the $ 1 $-dimensional observables be
\[{F_1}\left( x \right) = \sum\limits_{k = 1}^{\infty} {\frac{1}{{{k^2}}}} \sin \left( {2k\pi x} \right), {F_2}\left( x \right) = \sin \left( {2\pi x} \right),\]
  the rotations be  $ {\rho _1} = 0.1001000100001 \cdots $ and $ {\rho _2} = 1 $, and the initial point be $ \theta=10^{-1} $. One notices that $ F_1(x) $ admits a regularity lower than $ C^2 (\mathbb{T}^1) $, $ F_2(x) $ is analytic, and $ \rho_1 $ is extremely Liouvillean (nearly rational). Then the joint nonresonant rotation becomes $ \tilde \rho  = \left( {0.1001000100001 \cdots ,1} \right) $. It does not belong to any Diophantine class. Indeed, it is also extremely Liouvillean, because for any fixed $ \gamma>0 $ and $ \tau>0 $, there exists a subsequence $ {\{ {p_{{l_s}}},{q_{{l_s}}}\} _{s \in {\mathbb{N}^ + }}} $ of the approximants $ {\left\{ {{p_l}/{q_l}} \right\}_{l \in {\mathbb{N}^ + }}} $ of $ \rho_1 $, such that for infinitely many $ k = \left( {{k_1},{k_2}} \right) = \left( {{p_{l_s}}, 0} \right) \in {\mathbb{Z}^2} $, it holds 
  \[\left| {k \cdot \tilde \rho  - {q_{{l_s}}}} \right| = \left| {{p_{{l_s}}}{\rho _1} - {q_{{l_s}}}} \right| < \frac{\gamma }{{\left|{p_{{l_s}}}\right|^\tau}} = \frac{\gamma }{{{{\left| k \right|}^\tau }}},\]
  which proves the claim. For simplicity, we do not give a specific expression of the averages as we did in  Subsection \ref{MULSUBSEC1}. It is evident that the balancing condition \eqref{jointcon1}  does not hold for any $ m \in \mathbb{N}^+ $, and one naturally expects a relatively slower rate of convergence for the multiple weighted average (although faster than the unweighted type). 
  
  Figure \ref{FIGMUL2} illustrates this phenomenon. To ensure accuracy in the calculation, we utilize the Fourier series of $ F_1(x) $ with $ 100 $ terms and truncate $ \rho_1 $ to $ 0.010010001 $. It is evident that the convergence rate of the weighted multiple average is still faster than that of the unweighted one, but significantly \textit{slower} than the more regular and irrational case constructed in Subsection \ref{MULSUBSEC1}. For instance, for $ N=100 $, the error here is close to $ 10^{-5} $, whereas in Figure \ref{FIGMUL2}, the error is close to $ 10^{-11} $.

\section{Proof of the  abstract main  results}\label{Proof of main results}
This section is devoted to proving  the  abstract main  results. It should be pointed out that the analysis of the continuous case $ \mathrm{CMW}_T^\ell $ is much  easier compared to the discrete case $ \mathrm{DMW}_N^\ell $, as direct integration by parts can be utilized without the need for the Poisson summation formula. Hence, we have omitted the proof in this paper.

\subsection{Arbitrary polynomial convergence in the finite-dimensional case: Proof of Theorem \ref{MT1}}\label{SubsecProT1}
Let us first consider the discrete case \eqref{MT1lisan}. Throughout the subsequent discussion, let $ C_1>0 $  denote a generic constant independent of $ N $, which may vary in the context. We stress that $ C_1 $ is also independent of the dimension $ d $ thanks to the  boundedness condition \eqref{jointcon1} we have proposed.

With \eqref{product}, we obtain that
\begin{align}
	&{F_1}\left( {\mathscr{T}_{{\rho _1}}^n\left( \theta  \right)} \right) \cdots {F_\ell }\left( {\mathscr{T}_{{\rho _\ell }}^n\left( \theta  \right)} \right) \notag\\
	= &\prod\limits_{j = 1}^\ell  {\left( {\sum\limits_{{k^j} \in {\mathbb{Z}^d}} {{{({{\hat F}_j})}_{{k^j}}}{e^{2\pi i{k^j} \cdot \left( {\theta  + n{\rho _j}} \right)}}} } \right)} \notag\\
	\label{constant} = &\prod\limits_{j = 1}^\ell  {\left( {\int_{{\mathbb{T}^d}} {{F_j}(\hat \theta )d\hat \theta } } \right)}  + \sum\limits_{0 \ne {k^1}, \ldots ,{k^\ell } \in \mathbb{Z}^d } {\left( {\prod\limits_{j = 1}^\ell  {{{({{\hat F}_j})}_{{k^j}}}} } \right)\exp \left( {2\pi i\left( { \sum\nolimits_{j = 1}^\ell  {{k^j}\cdot \theta}   + n\sum\nolimits_{j = 1}^\ell  {{k^j} \cdot {\rho _j}} } \right)} \right)} .
\end{align}

Next, we need to address the challenges arising from small divisors. For the fixed number $ {\left( {\sum\nolimits_{j = 1}^\ell  {{k^j}} } \right) \cdot \tilde \rho } $, denote by $ \tilde{n} \in \mathbb{N} $  the closest integer to it. Note that $\tilde{n}$ is unique due to our Finite-dimensional joint  nonresonant condition proposed in Definition  \ref{Fi}. On the one hand, we have
\begin{align}
	{\left| {\sum\nolimits_{j = 1}^\ell  {{k^j} \cdot {\rho _j}}  - \tilde n} \right|^{ - m}} 	&= {\left| {\left( {\sum\nolimits_{j = 1}^\ell  {{k^j}} } \right) \cdot \tilde \rho  - \tilde n} \right|^{ - m}}\notag\\
	& \leqslant {\alpha ^{ - m}}{\Delta ^{  m}}\left( {\left\| {\sum\nolimits_{j = 1}^\ell  {{k^j}} } \right\|} \right)\notag\\
	\label{prfT1-1}& \equiv {\alpha ^{ - m}}{\Delta ^m}\left( {\sum\nolimits_{j = 1}^\ell  {\left\| {{k^j}} \right\|} } \right).
\end{align}
On the other hand,  with $ m \geqslant 2$ we have
\begin{equation}\label{prfT1-2}
		\sum\limits_{\tilde n \ne n \in \mathbb{Z}} {{{\left| {\sum\nolimits_{j = 1}^\ell  {{k^j} \cdot {\rho _j}}  - n} \right|}^{ - m}}}  \leqslant 2\sum\limits_{n = 0}^{ + \infty } {{{\left( {n + {2^{ - 1}}} \right)}^{ - m}}} \leqslant 2\sum\limits_{n = 0}^{ + \infty } {{{\left( {n + {2^{ - 1}}} \right)}^{ - 2}}} < 10.
\end{equation}
 Note that $ \left\| {\sum\nolimits_{j = 1}^\ell  {{k^j}} } \right\| \geqslant 1 $ since $ \sum\nolimits_{j = 1}^\ell  {{k^j}}  = \left( {{k^1}, \ldots ,{k^\ell }} \right) \in {\mathbb{Z}^{d\ell }}\backslash \left\{ 0 \right\} $, and $ {\Delta ^{  m}}\left( 1 \right) = 1 $. We therefore obtain that
\begin{align}
	\sum\limits_{n =  - \infty }^{ + \infty } {{{\left| {\sum\nolimits_{j = 1}^\ell  {{k^j} \cdot {\rho _j}}  - n} \right|}^{ - m}}}  
	&= {\left| {\sum\nolimits_{j = 1}^\ell  {{k^j} \cdot {\rho _j}}  - \tilde n} \right|^{ - m}} + \sum\limits_{\tilde n \ne n \in \mathbb{Z}} {{{\left| {\sum\nolimits_{j = 1}^\ell  {{k^j} \cdot {\rho _j}}  - n} \right|}^{ - m}}} \notag \\
	&\leqslant {C_1}{\Delta ^{  m}}\left( {\sum\nolimits_{j = 1}^\ell  {\left\| {{k^j}} \right\|} } \right).\notag
\end{align}
This leads to
\begin{align}
	&\sum\limits_{n =  - \infty }^{ + \infty } {\left| {\int_0^1 {w\left( y \right)\exp \left( {2\pi iNy\left( {\sum\nolimits_{j = 1}^\ell  {{k^j} \cdot {\rho _j}}  - n} \right)} \right)dy} } \right|} \notag \\
	\label{prfT1-3}\leqslant &{\left( {2\pi N} \right)^{ - m}}{\left\| {{w^{\left( m \right)}}} \right\|_{{L^1}\left( {0,1} \right)}}\sum\limits_{n =  - \infty }^{ + \infty } {{{\left| {\sum\nolimits_{j = 1}^\ell  {{k^j} \cdot {\rho _j}}  - n} \right|}^{ - m}}}  \\
	\label{sumite} \leqslant &{C_1}{N^{-m}}{\Delta ^m}\left( {\sum\nolimits_{j = 1}^\ell  {\left\| {{k^j}} \right\|} } \right)
\end{align}
by integrating by parts since $ w \in C_0^{\infty}([0,1]) $, where the following trivial estimate is used
\begin{equation}\label{MUL6.6}
	{\left| {\int_0^1 {{w^{\left( m \right)}}\left( y \right)\exp \left( {2\pi iNy\left( {\sum\nolimits_{j = 1}^\ell  {{k^j} \cdot {\rho _j}}  - n} \right)} \right)dy} } \right|}\leqslant{\left\| {{w^{\left( m \right)}}} \right\|_{{L^1}\left( {0,1} \right)}}.
\end{equation}
At this time, with the help of \eqref{sumite} and the Poisson summation formula in Lemma \ref{po}, we arrive at the estimates below due to $ w \in C_0^{\infty}([0,1]) $:
\begin{align}
	&\left| {{N^{ - 1}}\sum\limits_{n = 0}^{N - 1} {w\left( {n/N} \right)\exp \left( {2\pi in\sum\nolimits_{j = 1}^\ell  {{k^j} \cdot {\rho _j}} } \right)} } \right|\notag \\
	=&\left| {{N^{ - 1}}\sum\limits_{n = -\infty}^{+\infty} {w\left( {n/N} \right)\exp \left( {2\pi in\sum\nolimits_{j = 1}^\ell  {{k^j} \cdot {\rho _j}} } \right)} } \right|\notag \\
	= &\left| {{N^{ - 1}}\sum\limits_{n =  - \infty }^{ + \infty } {\int_{ - \infty }^{ + \infty } {w\left( {t/N} \right)\exp \left( {2\pi it\left( {\sum\nolimits_{j = 1}^\ell  {{k^j} \cdot {\rho _j}}  - n} \right)} \right)dt} } } \right|\notag \\
	= &\left| {{N^{ - 1}}\sum\limits_{n =  - \infty }^{ + \infty } {\int_{ 0 }^{ N } {w\left( {t/N} \right)\exp \left( {2\pi it\left( {\sum\nolimits_{j = 1}^\ell  {{k^j} \cdot {\rho _j}}  - n} \right)} \right)dt} } } \right|\notag \\
	\leqslant &\sum\limits_{n =  - \infty }^{ + \infty } {\left| {\int_0^1 {w\left( y \right)\exp \left( {2\pi iNy\left( {\sum\nolimits_{j = 1}^\ell  {{k^j} \cdot {\rho _j}}  - n} \right)} \right)dy} } \right|} \notag \\
	\label{sigmajifen} \leqslant &{C_1}{N^{-m}}{\Delta ^m}\left( {\sum\nolimits_{j = 1}^\ell  {\left\| {{k^j}} \right\|} } \right).
\end{align}

Now, building upon the previous preparations and 
\[\frac{N}{{{A_N}}} = {\left( {\frac{1}{N}\sum\limits_{s = 0}^{N - 1} {w\left( {s/N} \right)} } \right)^{ - 1}} \to {\left( {\int_0^1 {w\left( x \right)dx} } \right)^{ - 1}} > 0,\]
 one can derive the promised polynomial convergence (of order $ m $) as
\begin{align}
	&{\left\| {\mathrm{DMW}_N^\ell \left( \mathcal{F} \right)\left( \theta  \right) - \prod\limits_{j = 1}^\ell  {\left( {\int_{{\mathbb{T}^d}} {{F_j}(\hat \theta )d\hat \theta } } \right)} } \right\|_\mathscr{B}}\notag \\
	\label{useconstant} = &{\left\| {\frac{1}{{{A_N}}}\sum\limits_{n = 0}^{N - 1} {w\left( {n/N} \right)\sum\limits_{0 \ne {k^1}, \ldots ,{k^\ell } \in {\mathbb{Z}^d}} {\left( {\prod\limits_{j = 1}^\ell  {{{({{\hat F}_j})}_{{k^j}}}} } \right)\exp \left( {2\pi i\left( { \sum\nolimits_{j = 1}^\ell  {{k^j}\cdot \theta}   + n\sum\nolimits_{j = 1}^\ell  {{k^j} \cdot {\rho _j}} } \right)} \right)} } } \right\|_\mathscr{B}} \\
	\label{useF} \leqslant &\frac{	C_1N}{{{A_N}}}\sum\limits_{0 \ne {k^1}, \ldots ,{k^\ell } \in {\mathbb{Z}^d}} {\frac{1}{{\prod\nolimits_{j = 1}^\ell  {{{\tilde \Delta }_j}\left( {\left\| {{k^j}} \right\|} \right)} }}\left| {{N^{ - 1}}\sum\limits_{n = 0}^{N - 1} {w\left( {n/N} \right)\exp \left( {2\pi in\sum\nolimits_{j = 1}^\ell  {{k^j} \cdot {\rho _j}} } \right)} } \right|}  \\
	\label{usesigmajifen} \leqslant &{C_1}{N^{ - m}}\sum\limits_{0 \ne {k^1}, \ldots ,{k^\ell } \in {\mathbb{Z}^d}} {\frac{{{\Delta ^m}\left( {\sum\nolimits_{j = 1}^\ell  {\left\| {{k^j}} \right\|} } \right)}}{{\prod\nolimits_{j = 1}^\ell  {{{\tilde \Delta }_j}\left( {\left\| {{k^j}} \right\|} \right)} }}}  \\
	\label{usejointcon1} =& \mathcal{O}\left( {{N^{ - m}}} \right).
\end{align}
Here \eqref{useconstant} uses \eqref{sumite}, \eqref{useF} is because  $ F_j \in \mathscr{B}_{\tilde \Delta_j } $ for $ 1 \leqslant j \leqslant \ell $,   \eqref{usesigmajifen} uses \eqref{sigmajifen}, and finally \eqref{usejointcon1} follows from the boundedness condition \eqref{jointcon1}. This gives the proof of the discrete case \eqref{MT1lisan}, i.e., the polynomial convergence $\mathcal{O}\left( {{N^{ - m}}}\right)$ of the multiple ergodic average $ \mathrm{DMW}_N^\ell $.

Given that the proof of the continuous case \eqref{MT1lianxu} follows a similar structure to the one presented above and is, in fact, simpler, we can conclude the proof of Theorem \ref{MT1}.

\subsection{Arbitrary polynomial convergence in the infinite-dimensional case:  Proof of Theorem \ref{MT2}}
The proof closely resembles that of Theorem \ref{MT1}, with the key observation being that the universal constant is dimension-independent, as indicated in Comment \textbf{(C3)}.

\subsection{Exponential convergence in the  finite-dimensional case:  Proof of Theorem \ref{MT3}}\label{MULSEC6.3} It suffices to show the proof of the discrete case \eqref{MT3lisan}. Firstly, note that
\begin{align*}
	&\left| {{N^{ - 1}}\sum\limits_{n = 0}^{N - 1} {w\left( {n/N} \right)\exp \left( {2\pi in\sum\nolimits_{j = 1}^\ell  {{k^j} \cdot {\rho _j}} } \right)} } \right|\\
	\leqslant &{N^{ - 1}}\sum\limits_{n = 0}^{N - 1} {w\left( {n/N} \right)}  \to \int_0^1 {w\left( x \right)dx}  = 1.
\end{align*}
Building upon the proof of Theorem \ref{MT1} and the definition of the truncated space $ \mathscr{S}(x) $ in \eqref{quasiS}, one derives that
\begin{align}
	&{\left\| {\mathrm{DMW}_N^\ell \left( F \right)\left( \theta  \right) - \prod\limits_{j = 1}^\ell  {\left( {\int_{{\mathbb{T}^d}} {{F_j}(\hat \theta )d\hat \theta } } \right)} } \right\|_\mathscr{B}}\notag \\
	& \leqslant {\tilde C_1}\sum\limits_{0 \ne {k^1}, \ldots ,{k^\ell } \in {\mathbb{Z}^d}} {\frac{1}{{\prod\nolimits_{j = 1}^\ell  {{{\tilde \Delta }_j}\left( {\left\| {{k^j}} \right\|} \right)} }}\left| {{N^{ - 1}}\sum\limits_{n = 0}^{N - 1} {w\left( {n/N} \right)\exp \left( {2\pi in\sum\nolimits_{j = 1}^\ell  {{k^j} \cdot {\rho _j}} } \right)} } \right|} \notag \\
	&    \leqslant {{\tilde C}_1}\sum\limits_{0 \ne \sum\nolimits_{j = 1}^\ell  {{k^j}}  \in \mathscr{S}\left( N \right)} {\frac{{\mathcal{I}\left( {N,\left\{ {{k^j}} \right\}_{j = 1}^\ell ,\left\{ {{\rho _j}} \right\}_{j = 1}^\ell } \right)}}{{\prod\nolimits_{j = 1}^\ell  {{{\tilde \Delta }_j}\left( {\left\| {{k^j}} \right\|} \right)} }}} \notag \\
	&\;\;\;\; + {{\tilde C}_1}\sum\limits_{0 \ne \sum\nolimits_{j = 1}^\ell  {{k^j}}  \in {\mathbb{Z}^{d\ell }}\backslash \mathscr{S}\left( N \right)} {\frac{1}{{\prod\nolimits_{j = 1}^\ell  {{{\tilde \Delta }_j}\left( {\left\| {{k^j}} \right\|} \right)} }}} \notag \\
	\label{J1+J2}&: = {\tilde C_1}\left( {{\mathcal{J}_1} + {\mathcal{J}_2}} \right),
\end{align}
where the constant $ \tilde C_1 $ is independent of the parameter $ m $ (the finite degree of integration by parts) in Theorem \ref{MT1}, and we denote
\begin{align}
	&\mathcal{I}\left( {N,\left\{ {{k^j}} \right\}_{j = 1}^\ell ,\left\{ {{\rho _j}} \right\}_{j = 1}^\ell } \right)\notag \\
	\label{IIIIII}	: = &\sum\limits_{n =  - \infty }^{ + \infty } {\left| {\int_0^1 {w\left( y \right)\exp \left( {2\pi iNy\left( {\sum\nolimits_{j = 1}^\ell  {{k^j} \cdot {\rho _j}}  - n} \right)} \right)dy} } \right|}
\end{align}
in $ {\mathcal{J}_1} $ for convenience.

Next we provide a summary of our strategy. In view of \eqref{J1+J2}, we will do different operations for $ \mathcal{J}_1 $ and $ \mathcal{J}_2 $. For the former, we consider letting the time  of integration by parts (specifically $ \mathcal{K}(N) $ chosen below) vary, i.e., adapting it based on $ N $ and the adaptive function $ \varphi $ which appears in the truncated space $ \mathscr{S}(x) $. This approach is expected to yield an exponential estimate for  $ \mathcal{J}_1 $. As for the latter, the truncated smallness condition \eqref{T3jointcon1} implies that $ \mathcal{J}_2 $ is automatically exponentially small. By combining these two parts one could complete  the proof. The details are given below.

On the one hand, for the absolute constant $ \beta>0 $ in Lemma \ref{GAODAO}, there exist $ \bar \beta  > \beta  $ and some $ \beta_1>0 $ such that the following holds with $ N $ sufficiently large:
\[{\left( {\frac{{{\mathcal{K}^\beta }\left( N \right)}}{{2\pi \alpha \varphi \left( N \right)}}} \right)^{\mathcal{K}\left( N \right)}} \leqslant {\left( {\frac{{{\mathcal{K}^{\bar \beta }}\left( N \right)}}{{\varphi \left( N \right)}}} \right)^{\mathcal{K}\left( N \right)}} = \mathcal{O}\left( {\exp \left( { - {\varphi ^{{\beta _1}}}\left( N \right)} \right)} \right),\]
where $ \mathcal{K}\left( N \right) = \left[ {{e^{ - 1}}{\varphi ^{{{\bar \beta }^{ - 1}}}}\left( N \right)} \right] \geqslant 2 $ due to $ \varphi(+\infty) =+\infty $ (note that $ \varphi $ is an adaptive function, see Definition \ref{AdaptivedEF}). Recall \eqref{prfT1-1}, \eqref{prfT1-2} and \eqref{prfT1-3}. Consequently, for $ 0 \ne \sum\nolimits_{j = 1}^\ell  {{k^j}}  \in \mathscr{S}\left( N \right) $ we have
\begin{align}
	&\mathcal{I}\left( {N,\left\{ {{k^j}} \right\}_{j = 1}^\ell ,\left\{ {{\rho _j}} \right\}_{j = 1}^\ell } \right)\notag\\
	\label{IIIIIIIIIII} = &{\left( {2\pi N} \right)^{ - \mathcal{K}\left( N \right)}}{\left\| {{w^{\left( {\mathcal{K}\left( N \right)} \right)}}} \right\|_{{L^1}\left( {0,1} \right)}}\left( {\sum\limits_{n =  - \infty }^{ + \infty } {{{\left| {\sum\nolimits_{j = 1}^\ell  {{k^j} \cdot {\rho _j} - n} } \right|}^{ - \mathcal{K}\left( N \right)}}} } \right)\\
	\leqslant &{\left( {2\pi N} \right)^{ - \mathcal{K}\left( N \right)}}{\left( {\mathcal{K}\left( N \right)} \right)^{\mathcal{K}\left( N \right)\beta }}\left( {{\alpha ^{ - \mathcal{K}\left( N \right)}}{\Delta ^{\mathcal{K}\left( N \right)}}\left( {\left\| {\sum\nolimits_{j = 1}^\ell  {{k^j}} } \right\|} \right)} \right)\notag\\
	\equiv& {\left( {\frac{{{\mathcal{K}^\beta }\left( N \right)}}{{2\pi \alpha N}}} \right)^{\mathcal{K}\left( N \right)}}\left( {{\Delta ^{\mathcal{K}\left( N \right)}}\left( {\sum\nolimits_{j = 1}^\ell  {\left\| {{k^j}} \right\|} } \right)} \right)\notag\\
	\leqslant &{\left( {\frac{{{\mathcal{K}^\beta }\left( N \right)}}{{2\pi \alpha N}}} \right)^{\mathcal{K}\left( N \right)}}\left( {{\Delta ^{\mathcal{K}\left( N \right)}}\left( {\sum\nolimits_{j = 1}^\ell  {{\ell ^{ - 1}}{\Delta ^{ - 1}}\left( {N/\varphi \left( N \right)} \right)} } \right)} \right)\notag\\
	= & {\left( {\frac{{{\mathcal{K}^\beta }\left( N \right)}}{{2\pi \alpha \varphi \left( N \right)}}} \right)^{\mathcal{K}\left( N \right)}}\notag\\
	\label{zhishuzhishu} = &\mathcal{O}\left( {\exp \left( { - {\varphi ^{{\beta _1}}}\left( N \right)} \right)} \right).
\end{align}
This leads to the estimates for $ \mathcal{J}_1 $ below:
\begin{align}
	{\mathcal{J}_1} &= \sum\limits_{0 \ne \sum\nolimits_{j = 1}^\ell  {{k^j}}  \in \mathscr{S}\left( N \right)} {\frac{{\mathcal{I}\left( {N,\left\{ {{k^j}} \right\}_{j = 1}^\ell ,\left\{ {{\rho _j}} \right\}_{j = 1}^\ell } \right)}}{{\prod\nolimits_{j = 1}^\ell  {{{\tilde \Delta }_j}\left( {\left\| {{k^j}} \right\|} \right)} }}} \notag \\
	& = \sum\limits_{0 \ne \sum\nolimits_{j = 1}^\ell  {{k^j}}  \in \mathscr{S}\left( N \right)} {\frac{{\mathcal{O}\left( {\exp \left( { - {\varphi ^{{\beta _1}}}\left( N \right)} \right)} \right)}}{{\prod\nolimits_{j = 1}^\ell  {{{\tilde \Delta }_j}\left( {\left\| {{k^j}} \right\|} \right)} }}}\notag \\
	& = \mathcal{O}\left( \sum\limits_{0 \ne {k^1}, \ldots ,{k^\ell } \in {\mathbb{Z}^d}} {\frac{1}{{\prod\nolimits_{j = 1}^\ell  {{{\tilde \Delta }_j}\left( {\left\| {{k^j}} \right\|} \right)} }}} \right) \cdot  \mathcal{O}\left( {\exp \left( { - {\varphi ^{{\beta _1}}}\left( N \right)} \right)} \right)\notag\\
	\label{J1-2}&= \mathcal{O}\left( {\exp \left( { - {\varphi ^{{\beta _1}}}\left( N \right)} \right)} \right),
\end{align}
where the truncated smallness condition \eqref{T3jointcon1} is applied in \eqref{J1-2} due to Cauchy's Theorem, as shown in Comment \textbf{(C1)}.

On the other hand, in view of the truncated smallness condition \eqref{T3jointcon1}, we directly get
\begin{equation}\label{J2-2}
	{\mathcal{J}_2} = \sum\limits_{0 \ne \sum\nolimits_{j = 1}^\ell  {{k^j}}  \in {\mathbb{Z}^{d\ell }}\backslash \mathscr{S}\left( N \right)} {\frac{1}{{\prod\nolimits_{j = 1}^\ell  {{{\tilde \Delta }_j}\left( {\left\| {{k^j}} \right\|} \right)} }}}  = \mathcal{O}\left( {{e^{ - cN}}} \right).
\end{equation}

Substituting \eqref{J1-2} and \eqref{J2-2} into \eqref{J1+J2} and comparing the order, we finally arrive at the desired exponential convergence for the multiple ergodic average $ \mathrm{DMW}_N^\ell $:
\[	{\left\| {\mathrm{DMW}_N^\ell\left( \mathcal{F} \right)\left( \theta  \right) - \prod\limits_{j = 1}^\ell  \left({\int_{{\mathbb{T}^d}} {{F_j}(\hat \theta )d\hat \theta } }\right) } \right\|_\mathscr{B}} =  \mathcal{O}\left( \exp(-\varphi^{\beta_1}(N)) \right),\]
because the adaptive function $ \varphi $ satisfies $ \varphi(x)=o(x) $, as defined in Definition \ref{AdaptivedEF}. 

As to the continuous case \eqref{MT3lianxu}, the proof is similar since one does not have to apply the Poisson summation formula. We therefore finish the proof of Theorem \ref{MT3}.

\subsection{Exponential convergence in the  infinite-dimensional case:  Proof of Theorem \ref{MT4}}
The proof closely resembles that of Theorem \ref{MT3}, with the key observation being that the universal constant is dimension-independent, as indicated in Comment \textbf{(C3)} due to our truncated smallness condition \eqref{T4jointcon1}.

\subsection{Exponential convergence  via trigonometric polynomials:  Proof of Theorem \ref{MT5}}
We only show the proof for the discrete case \eqref{MT51} with $ d <+\infty $. Denote by $ C_2>0 $ a universal constant independent of $ N $. With the analysis in Section \ref{MULSEC6.3} in mind, one verifies that
\[\sum\limits_{n =  - \infty }^{ + \infty } {{{\left| {\sum\nolimits_{j = 1}^\ell  {{k^j} \cdot {\rho _j} - n} } \right|}^{ - \tilde {\mathcal{K}}\left( N \right)}}}  \leqslant {C_2}\]
since $ 0 \ne \| k^j \| \leqslant K $. Choose $\tilde {\mathcal{K}}\left( N \right) = \left[ {{e^{ - 1}}{{\left( {C_2^{ - 1}N} \right)}^{1/\beta }}} \right] \geqslant 2$. Then it follows from \eqref{IIIIII} and \eqref{IIIIIIIIIII} that
\begin{align}
	&\mathcal{I}\left( {N,\left\{ {{k^j}} \right\}_{j = 1}^\ell ,\left\{ {{\rho _j}} \right\}_{j = 1}^\ell } \right)\notag \\
	\leqslant &{\left( {2\pi N} \right)^{ - \tilde {\mathcal{K}}\left( N \right)}}{\left( {\tilde {\mathcal{K}}\left( N \right)} \right)^{\tilde {\mathcal{K}}\left( N \right)\beta }}\left( {\sum\limits_{n =  - \infty }^{ + \infty } {{{\left| {\sum\nolimits_{j = 1}^\ell  {{k^j} \cdot {\rho _j} - n} } \right|}^{ - {\mathcal{K}}\left( N \right)}}} } \right)\notag \\
	\leqslant &{\left( {{C_2}{N^{ - 1}}{{\tilde {\mathcal{K}}}^\beta }\left( N \right)} \right)^{\tilde {\mathcal{K}}\left( N \right)}}\notag \\
	\label{MT5--1} = &\mathcal{O}\left( {\exp \left( {{N^{ - \hat c}}} \right)} \right)
\end{align}
with some $ \hat{c}>0 $, in other words, \textit{we obtain a better convergence rate than that in \eqref{zhishuzhishu} under the trigonometric polynomial setting.} One finally arrives at the followings by \eqref{MT5--1}:
\begin{align*}
	&{\left\| {\mathrm{DMW}_N^\ell \left( \mathcal{F} \right)\left( \theta  \right) - \prod\limits_{j = 1}^\ell  {\left( {\int_{{\mathbb{T}^1}} {{F_j}(\hat \theta )d\hat \theta } } \right)} } \right\|_\mathscr{B}}\\
	= &\mathcal{O}\left( {\sum\limits_{0 \ne |{k^\iota }| \leqslant K,1 \leqslant \iota  \leqslant j} {\frac{{\mathcal{I}\left( {N,\left\{ {{k^j}} \right\}_{j = 1}^\ell ,\left\{ {{\rho _j}} \right\}_{j = 1}^\ell } \right)}}{{\prod\nolimits_{j = 1}^\ell  {{{\tilde \Delta }_j}\left( {\left| {{k^j}} \right|} \right)} }}} } \right)\\
	= &\mathcal{O}\left( {\sum\limits_{0 \ne |{k^\iota }| \leqslant K,1 \leqslant \iota  \leqslant j} {\frac{1}{{\prod\nolimits_{j = 1}^\ell  {{{\tilde \Delta }_j}\left( {\left| {{k^j}} \right|} \right)} }}} } \right) \cdot \mathcal{O}\left( {\exp \left( {{N^{ - \hat c}}} \right)} \right)\\
	= &\mathcal{O}\left( {\exp \left( {{N^{ - \hat c}}} \right)} \right).
\end{align*}
This demonstrates the exponential convergence. As to the continuous case \eqref{MT52}, the proof is similar and we therefore omit here.

For the continuous case with $ d=1 $, i.e., \eqref{MT53}, we only have to notice that
\[\sum\limits_{0 \ne {k^1}, \ldots ,{k^\ell } \in \mathbb{Z}} {\frac{1}{{\prod\nolimits_{j = 1}^\ell  {{{\tilde \Delta }_j}\left( {\left| {{k^j}} \right|} \right)} }}}  = \mathcal{O}\left( {\prod\limits_{j = 1}^\ell  {\left( {\int_1^{ + \infty } {\frac{1}{{{{\tilde \Delta }_j}\left( x \right)}}dx} } \right)} } \right) = \mathcal{O}\left( 1 \right)\]
thanks to \eqref{MT54}, then the proof is also similar since the universal coefficient is bounded, and the estimates obtained by integration by parts are exponentially small via the  finite nonresonant assumptions, namely $\sum\nolimits_{j = 1}^\ell  {{k^j}{\rho _j}}  \ne 0$ hold for all $0 < \left\| {{k^j}} \right\| \leqslant K$. This proves Theorem \ref{MT5}.

\section{Appendix}
\begin{lemma}[Poisson summation formula] \label{po}
			Let $ h $ be a continuous function on
			$ \mathbb{R}^n $ which satisfies for some $ C,\delta>0 $ and for all $ x \in \mathbb{R}^n $
			\[\left| {h\left( x \right)} \right| \leqslant C{\left( {1 + x} \right)^{ - n - \delta }},\]
			and whose Fourier transform $ {\hat h} $ restricted on $ \mathbb{Z}^n $ satisfies
			\[\sum\limits_{m \in {\mathbb{Z}^n}} {|\hat h\left( m \right)|}  <  + \infty .\]
			Then for all 	$ x\in \mathbb{R}^n $ we have
			\[\sum_{m \in \mathbb{Z}^n} \hat{h}(m) e^{2 \pi i m \cdot x}=\sum_{k \in \mathbb{Z}^n} h(x+k),\]
			and in particular
			\[\sum_{m \in \mathbb{Z}^n} \hat h(m)=\sum_{k \in \mathbb{Z}^n} h(k).\]
		\end{lemma}
		\begin{proof}
			See Chapter 3 in \cite{MR3243734} for  details. See also Chapter 3  in  \cite{MR1970295} for another stronger version in the Schwartz space $ \mathcal{C}(\mathbb{R}) $.
\end{proof}

\begin{lemma}\label{GAODAO}
	Define the weighting function as
	\[	w\left( x \right) = \exp \left( { - {x^{ - 1}}{{\left( {1 - x} \right)}^{ - 1}}} \right) \cdot {\left( {\int_0^1 {\exp \left( { - {t^{ - 1}}{{\left( {1 - t} \right)}^{ - 1}}} \right)\mathrm{d}t} } \right)^{ - 1}}\]
	on $ (0,1) $. Then there exists $ \beta>0 $ such that
	\begin{equation}\label{L1}
		{\big\| {{w^{\left( n \right)}}} \big\|_{{L^1}\left( {0,1} \right)}} = \mathcal{O}\big( {{n^{n\beta }}} \big),\;\;n \to  + \infty .
	\end{equation}
\end{lemma}
\begin{proof}See Lemma 5.3 in \cite{MR4768308} for details.
\end{proof}

\begin{lemma}\label{5.2}
	For any given $ \rho_*  > 0$ and $\mu  \in {\mathbb{N}^ + } $, there exists $ \tau  = \tau \left( {\eta ,\mu } \right) > 0 $ such that
	\[ \prod\limits_{j \in \mathbb{N}} {\left( {1 + {{\left| {{k_j}} \right|}^\mu }{{\left\langle j \right\rangle }^\mu }} \right)}  \leqslant \exp \left( {\frac{\tau }{{{\rho_* ^{1/\eta }}}}\log \left( {\frac{\tau }{\rho_* }} \right)} \right) \cdot {e^{\rho_* {{\left| k \right|}_\eta }}}.\]
\end{lemma}
\begin{proof}
	See details in Lemma B.1 in \cite{MR4201442} and  Lemma 7.2 in \cite{MR4091501}.
\end{proof}	

\begin{lemma}\label{wuqiongweijishu}
	For $ 2 \leqslant \eta  \in {\mathbb{N}^ + } $, the following holds whenever $ \nu \in \mathbb{N}^+ $ is sufficiently large:
	\begin{equation}\label{lemma121}
		\sum\limits_{0 \ne k \in \mathbb{Z}_ * ^\infty ,{{\left| k \right|}_\eta } = \nu } 1  := \# \left\{ {k:0 \ne k \in \mathbb{Z}_ * ^\infty ,{{\left| k \right|}_\eta } = \nu  \in {\mathbb{N}^ + }} \right\}=\mathcal{O}\left({\nu ^{{\nu ^{1/\eta }}}}\right).
	\end{equation}
\end{lemma}
\begin{proof}
	Recall that
	\begin{equation}\notag
		{\left| k \right|_\eta } = \sum\limits_{j \in \mathbb{N}} {{{\left\langle j \right\rangle }^\eta }\left| {{k_j}} \right|}  \in {\mathbb{N}^ + }.
	\end{equation}
	Then the  largest non-zero integer $ {j_{\max }} $ in \eqref{lemma121}  satisfies that $ {j_{\max }} \leqslant \left[ {{\nu ^{1/\eta }}} \right] $. Therefore, \begin{align}
		&\# \left\{ {k:0 \ne k \in \mathbb{Z}_ * ^\infty ,{{\left| k \right|}_\eta } = \nu  \in {\mathbb{N}^ + }} \right\} \notag \\
		\leqslant &\# \left\{ {k:0 \ne k \in \mathbb{Z}_ * ^\infty ,|{k_0}| + |{k_1}| +  \cdots  + |{k_{\left[ {{\nu ^{1/\eta }}} \right]}}| = \nu  \in {\mathbb{N}^ + }} \right\}\notag \\
		\leqslant &{2^{\left[ {{\nu ^{1/\eta }}} \right] + 1}} \cdot \# \left\{ {k:0 \ne k \in \mathbb{Z}_ * ^\infty ,\text{$ {k_j} \in \mathbb{N} $ for all $ j \in \mathbb{N} $},{k_0} + {k_1} +  \cdots  + {k_{\left[ {{\nu ^{1/\eta }}} \right]}} = \nu  \in {\mathbb{N}^ + }} \right\}\notag \\
		= &{2^{\left[ {{\nu ^{1/\eta }}} \right] + 1}} \cdot C_{\nu  + \left[ {{\nu ^{1/\eta }}} \right]}^\nu \notag \\
		\label{CCCC}\leqslant &{2^{\left[ {{\nu ^{1/\eta }}} \right] + 1}} \cdot {C_\eta }\frac{1}{{\sqrt {{\nu ^{1/\eta }}} }} \cdot {\nu ^{\left( {1 - 1/\eta } \right)\left( {{\nu ^{1/\eta }} + 1} \right)}} \cdot {e^{\left[ {{\nu ^{1/\eta }}} \right]}} \\
		\leqslant &  {C_\eta }{\nu ^{{\nu ^{1/\eta }}}}.\notag
	\end{align}
	Here \eqref{CCCC} uses the following fact, by applying the  Stirling's approximation $ n! \sim \sqrt {2\pi n} {\left( {n/e} \right)^n}  $ as $ n \to +\infty $:
	\begin{align*}
		&\;\;\;\;\;\;C_{\nu  + \left[ {{\nu ^{1/\eta }}} \right]}^\nu \\
		&= \frac{{\left( {\nu  + \left[ {{\nu ^{1/\eta }}} \right]} \right)!}}{{\nu !\left( {\left[ {{\nu ^{1/\eta }}} \right]} \right)!}} \sim \frac{{\sqrt {2\pi \left( {\nu  + \left[ {{\nu ^{1/\eta }}} \right]} \right)} {{\left( {\frac{{\nu  + \left[ {{\nu ^{1/\eta }}} \right]}}{e}} \right)}^{\nu  + \left[ {{\nu ^{1/\eta }}} \right]}}}}{{\sqrt {2\pi \nu } {{\left( {\frac{\nu }{e}} \right)}^\nu } \cdot \sqrt {2\pi \left[ {{\nu ^{1/\eta }}} \right]} {{\left( {\frac{{\left[ {{\nu ^{1/\eta }}} \right]}}{e}} \right)}^{\left[ {{\nu ^{1/\eta }}} \right]}}}}\\
		& \sim \frac{1}{{\sqrt {2\pi {\nu ^{1/\eta }}} }} \cdot {\left( {1 + \frac{{\left[ {{\nu ^{1/\eta }}} \right]}}{\nu }} \right)^\nu } \cdot {\left( {\frac{\nu }{{\left[ {{\nu ^{1/\eta }}} \right]}}} \right)^{\left[ {{\nu ^{1/\eta }}} \right]}} \cdot {\left( {1 + \frac{{\left[ {{\nu ^{1/\eta }}} \right]}}{\nu }} \right)^{\left[ {{\nu ^{1/\eta }}} \right]}}\\
		& = \frac{1}{{\sqrt {2\pi {\nu ^{1/\eta }}} }} \cdot {\left( {\frac{\nu }{{\left[ {{\nu ^{1/\eta }}} \right]}}} \right)^{\left[ {{\nu ^{1/\eta }}} \right]}} \cdot \exp \left( {\nu \log \left( {1 + \frac{{\left[ {{\nu ^{1/\eta }}} \right]}}{\nu }} \right)} \right) \\
		&\;\;\;\;\cdot \exp \left( {\left[ {{\nu ^{1/\eta }}} \right]\log \left( {1 + \frac{{\left[ {{\nu ^{1/\eta }}} \right]}}{\nu }} \right)} \right)\\
		& = \frac{1}{{\sqrt {2\pi {\nu ^{1/\eta }}} }} \cdot {\left( {\frac{\nu }{{\left[ {{\nu ^{1/\eta }}} \right]}}} \right)^{\left[ {{\nu ^{1/\eta }}} \right]}} \cdot \exp \left( {\nu \left( {\frac{{\left[ {{\nu ^{1/\eta }}} \right]}}{\nu } - \frac{1}{2}\frac{{{{\left[ {{\nu ^{1/\eta }}} \right]}^2}}}{{{\nu ^2}}} +  \cdots } \right)} \right) \\
		&\;\;\;\;\cdot \exp \left( {\left[ {{\nu ^{1/\eta }}} \right]\left( {\frac{{\left[ {{\nu ^{1/\eta }}} \right]}}{\nu } +  \cdots } \right)} \right)\\
		& = \frac{1}{{\sqrt {2\pi {\nu ^{1/\eta }}} }} \cdot {\left( {\frac{\nu }{{\left[ {{\nu ^{1/\eta }}} \right]}}} \right)^{\left[ {{\nu ^{1/\eta }}} \right]}} \cdot \exp \left( {\left[ {{\nu ^{1/\eta }}} \right] - \frac{{{{\left[ {{\nu ^{1/\eta }}} \right]}^2}}}{{2\nu }} +  \cdots } \right) \cdot \exp \left( {\frac{{{{\left[ {{\nu ^{1/\eta }}} \right]}^2}}}{\nu } +  \cdots } \right)\\
		& = \frac{1}{{\sqrt {2\pi {\nu ^{1/\eta }}} }} \cdot {\left( {\frac{\nu }{{\left[ {{\nu ^{1/\eta }}} \right]}}} \right)^{\left[ {{\nu ^{1/\eta }}} \right]}} \cdot \exp \left( {\left[ {{\nu ^{1/\eta }}} \right] + \mathcal{O}\left( 1 \right)} \right) \cdot \exp \left( {\mathcal{O}\left( 1 \right)} \right) (\text{since $ \eta \geqslant 2 $})\\
		& \leqslant {C_\eta }\frac{1}{{\sqrt {{\nu ^{1/\eta }}} }} \cdot {\nu ^{\left( {1 - 1/\eta } \right)\left( {{\nu ^{1/\eta }} + 1} \right)}} \cdot {e^{\left[ {{\nu ^{1/\eta }}} \right]}}.
	\end{align*}
	This proves the lemma.
\end{proof}

\begin{lemma}\label{tao}
	Let $ \zeta>0,\ell \in \mathbb{N}^+ $ and $ \left\{ {{a_j}} \right\}_{j = 1}^\ell  > 0 $ be given. Then there exists a constant $ C\left( {\zeta ,\ell } \right)>0 $ such that
	\[{\left( {\sum\nolimits_{j = 1}^\ell  {{a_j}} } \right)^\zeta } \leqslant C\left( {\zeta ,\ell } \right)\left( {\sum\nolimits_{j = 1}^\ell  {a_j^\zeta } } \right).\]
\end{lemma}
\begin{proof}
It suffices to prove the conclusion for $ \ell\geqslant 2 $. For $ \ell =2 $,
	denote $ {b_j} = {\left( {\sum\nolimits_{s = 1}^2  {{a_s}} } \right)^{ - 1}}{a_j} $ with  $ 1 \leqslant j \leqslant 2 $. This leads to
	$ \sum\nolimits_{j = 1}^2  {{b_j}}  = 1 $. Then we only need to verify that $ \sum\nolimits_{j = 1}^2  {b_j^\zeta }  \geqslant {C^{ - 1}}\left( {\zeta ,2 } \right)>0 $ for $ \zeta >0 $. It can be obtained by analyzing the monotonicity of the function $ h\left( x \right) = {x^\zeta } + {\left( {1 - x} \right)^\zeta } $ on the interval $ [0,1] $ (note that $ \zeta $ needs to be classified, i.e., $ 0<\zeta<1, \zeta=1 $ and $ \zeta>1 $):
	\[\sum\nolimits_{j = 1}^2 {b_j^\zeta }  \geqslant \min \left\{ {h\left( 0 \right),h\left( {{2^{ - 1}}} \right),h\left( 1 \right)} \right\} = \min \left\{ {1,{2^{1 - \zeta }}} \right\}: = {C^{ - 1}}\left( {\zeta ,2} \right).\]
	As to the case $ \ell \geqslant 3 $, letting $ C\left( {\zeta ,\ell } \right): = \max \left\{ {C\left( {\zeta ,2} \right),{{\left( {C\left( {\zeta ,2} \right)} \right)}^{\ell  - 1}}} \right\} $ yields that
	\begin{align*}
		{\left( {\sum\nolimits_{j = 1}^\ell  {{a_j}} } \right)^\zeta } &= {\left( {{a_1} + \sum\nolimits_{j = 2}^\ell  {{a_j}} } \right)^\zeta } \leqslant C\left( {\zeta ,2} \right)\left( {a_1^\zeta  + {{\left( {\sum\nolimits_{j = 2}^\ell  {{a_j}} } \right)}^\zeta }} \right)\\
		& \leqslant  \cdots  \leqslant \max \left\{ {C\left( {\zeta ,2} \right),{{\left( {C\left( {\zeta ,2} \right)} \right)}^{\ell  - 1}}} \right\}\left( {\sum\nolimits_{j = 1}^\ell  {a_j^\zeta } } \right) = C\left( {\zeta ,\ell } \right)\left( {\sum\nolimits_{j = 1}^\ell  {a_j^\zeta } } \right),
	\end{align*}
	as promised.
\end{proof}

 \section*{Acknowledgements} 
 Z. Tong and Y. Li would like to express their gratitude to Prof. Aihua Fan (Universit\'e de Picardie) for his valuable suggestions and comments, which led to a significant improvement of the paper.  Z. Tong  also extends sincere thanks to Prof. Valery V. Ryzhikov (Moscow State University) for his insightful discussions on the weighted rates of general dynamical systems, and to  Profs. Aleksandr G. Kachurovski\u{\i} (Sobolev Institute of Mathematics) and Ivan V. Podvigin (Sobolev Institute of Mathematics) for their invaluable assistance and discussions.  Z. Tong  was supported by the China Postdoctoral Science Foundation (Grant No. 2025M783102). Y. Li was supported in part by the National Natural Science Foundation of China (Grant Nos. 12071175, 12471183 and 12531009).


\begin{thebibliography}{DSSY17}

	
	

	

	
	
	\bibitem[BMP20]{MR4091501}
	L. Biasco, J. Massetti, M. Procesi,
	An abstract {B}irkhoff normal form theorem and exponential type stability of the 1d {NLS}.
	Comm. Math. Phys. 375 (2020), pp.~2089--2153. \url{https://doi.org/10.1007/s00220-019-03618-x}
	
	\bibitem[Bir31]{Birkhoff}
	G. D. Birkhoff,  Proof of the ergodic theorem. Proc Natl Acad 	Sci USA 17(12) (1931), pp.~656--660.




	\bibitem[Bou90]{MR1037434}
	J. Bourgain,
	Double recurrence and almost sure convergence.
	J. Reine Angew. Math. 404 (1990), pp.~140--161. \url{https://doi.org/10.1515/crll.1990.404.140}
	
	
	
	\bibitem[Bou05]{MR2180074}
	J. Bourgain,
	On invariant tori of full dimension for 1{D} periodic {NLS}.
	J. Funct. Anal. 229 (2005), pp.~62--94. \url{https://doi.org/10.1016/j.jfa.2004.10.019}
	

	
	\bibitem[DSSY17]{MR3718733}
	S. Das, Y. Saiki, E. Sander, J. A. Yorke,  Quantitative quasiperiodicity. Nonlinearity 30 (2017), pp.~4111--4140.
	\url{https://doi.org/10.1088/1361-6544/aa84c2}
	
	\bibitem[DY18]{MR3755876}
	S. Das, J. A. Yorke, Super convergence of ergodic averages for quasiperiodic orbits. Nonlinearity 31 (2018), pp.~491--501. \url{https://doi.org/10.1088/1361-6544/aa99a0}
	
	
	\bibitem[dR79]{MR553340}
	A. del Junco, J. Rosenblatt,  
	Counterexamples in ergodic theory and number theory.
	Math. Ann. 245 (1979), pp.~185--197. \url{https://doi.org/10.1007/BF01673506}
	
	



	
	\bibitem[DM23]{MR4582163}
	N. Duignan, J. D. Meiss, Distinguishing between regular and chaotic orbits of flows by the weighted Birkhoff average, Phys. D 449 (2023), Paper No. 133749, pp.~16. \url{https://doi.org/10.1016/j.physd.2023.133749}
	
	

	
	\bibitem[Fan17]{MR3693506}
	A.   Fan, 
	Fully oscillating sequences and weighted multiple ergodic limit. 
	C. R. Math. Acad. Sci. Paris 355 (2017), pp.~866--870.
\url{https://doi.org/10.1016/j.crma.2017.07.008}
	
	\bibitem[Fan18]{Fanarxiv}
	A.  Fan, 
	Oscillating sequences of higher orders and topological systems of quasi-discrete spectrum. Preprint 2018.
\url{https://doi.org/10.48550/arXiv.1802.052044}

\bibitem[Fan19]{Fan19}
A.   Fan,  
Weighted Birkhoff ergodic theorem with oscillating weights. 
Ergodic Theory Dynam. Systems 39 (2019), pp.~1275--1289. \url{https://doi.org/10.1017/etds.2017.81}
	
	\bibitem[FLW18]{MR3848423}
	A.  Fan,  L.   Liao, M. Wu, 
	Multifractal analysis of some multiple ergodic averages in linear cookie-cutter dynamical systems. 
	Math. Z. 290 (2018),  pp.~63--81. \url{https://doi.org/10.1017/etds.2017.81}
	
	

	
	
	\bibitem[FSW16]{MR3488037}
	A.  Fan,  J. Schmeling, M. Wu,
	Multifractal analysis of some multiple ergodic averages.
	Adv. Math. 295 (2016), pp.~271--333. \url{https://doi.org/10.1016/j.aim.2016.03.012}
	
	\bibitem[Fur77]{MR0498471}
	H. Furstenberg,
	Ergodic behavior of diagonal measures and a theorem of Szemer\'{e}di on arithmetic progressions.
	J. Analyse Math. 31 (1977), pp.~204--256. \url{https://doi.org/10.1007/BF02813304}
	
	
	
	
	\bibitem[Fur81]{MR0603625}
	H. Furstenberg,
	Recurrence in ergodic theory and combinatorial number theory. M. B. Porter Lectures. Princeton University Press, Princeton, N.J., 1981, pp.~xi+203.
	

	
	\bibitem[Gra14]{MR3243734}
	L. Grafakos,
	Classical Fourier analysis. Third edition. Graduate Texts in Mathematics, 249. Springer, New York, 2014, pp.~xviii+638. 
\url{https://doi.org/10.1007/978-1-4939-1194-3}
	
	\bibitem[Her79]{MR538680}
	M.-R. Herman, Sur la conjugaison diff\'{e}rentiable des diff\'{e}omorphismes du cercle \`a des rotations. Inst. Hautes \'{E}tudes Sci. Publ. Math. No. 49 (1979), pp.~5--233. \url{http://www.numdam.org/item?id=PMIHES_1979__49__5_0}
	

	
	\bibitem[HSY19a]{MR4041103}
	W. Huang, S. Shao, X.  Ye,
	Pointwise convergence of multiple ergodic averages and strictly ergodic models. J. Anal. Math. 139 (2019), pp.~265--305.
	\url{https://doi.org/10.1007/s11854-019-0061-3}

\bibitem[HSY19b]{MR3996054}
	W. Huang, S. Shao, X.   Ye,
Topological correspondence of multiple ergodic averages of nilpotent group actions.
J. Anal. Math. 138 (2019), pp.~687--715. \url{https://doi.org/10.1007/s11854-019-0036-4}

\bibitem[Kac96]{MR1422228}
A. G. Kachurovski\u{\i},  Rates of convergence in ergodic theorems. (Russian) Uspekhi Mat. Nauk 51 (1996), pp.~73--124; translation in Russian Math. Surveys 51 (1996), pp.~653--703. \url{https://doi.org/10.1070/RM1996v051n04ABEH002964}




\bibitem[KP16]{MR3643963}
A. G. Kachurovski\u{\i}, I. V. Podvigin,  Estimates of the rate of convergence in the von Neumann and Birkhoff ergodic theorems. Trans. Moscow Math. Soc. 2016, pp.~1--53. \url{https://doi.org/10.1090/mosc/256}

\bibitem[KP19]{MR3981324}
A. G. Kachurovski\u{\i}, I. V. Podvigin, 
On measuring the rate of convergence in the Birkhoff ergodic theorem. (Russian)
Mat. Zametki 106 (2019), pp.~40--52; translation in
Math. Notes 106 (2019), pp.~52--62. \url{https://doi.org/10.4213/mzm11817}

\bibitem[KP81]{MR612450}
S. Kakutani, K. Petersen,   The speed of convergence in the ergodic theorem. Monatsh. Math. 91 (1981), pp.~11--18. \url{https://doi.org/10.1007/BF01306954}


	



	
	\bibitem[Kre78]{MR0510630}
	U. Krengel,
	On the speed of convergence in the ergodic theorem.
	Monatsh. Math. 86 (1978/79),  pp.~3--6. \url{https://doi.org/10.1007/BF01300052}
	
	\bibitem[Las93a]{MR1234445}
	J. Laskar,
	Frequency analysis for multi-dimensional systems. Global dynamics and diffusion.
	Phys. D 67 (1993),  pp.~257--281. \url{https://doi.org/10.1016/0167-2789(93)90210-R}
	
	
	
	\bibitem[Las93b]{MR1222935}
	J. Laskar,
	Frequency analysis of a dynamical system.  Qualitative and quantitative behaviour of planetary systems (Ramsau, 1992). Celestial Mech. Dynam. Astronom. 56 (1993),  pp.~191--196.
	\url{https://doi.org/10.1007/BF00699731}
	
	
	\bibitem[Las99]{MR1720890}
	J. Laskar,
	Introduction to frequency map analysis.  Hamiltonian systems with three or more degrees of freedom ({S}'{A}gar\'{o}, 1995), pp.~134--150,	NATO Adv. Sci. Inst. Ser. C: Math. Phys. Sci., 533, Kluwer Acad. Publ., Dordrecht, 1999.
	
	
	
	\bibitem[Mac74]{MR0346131}
	G. W. Mackey,
	Ergodic theory and its significance for statistical mechanics and probability theory.
	Advances in Math. 12 (1974), pp.~178--268. \url{https://doi.org/10.1016/S0001-8708(74)80003-4}
	
	
	\bibitem[MS21]{MR4322369}
	J. D. Meiss,  E. Sander,
	Birkhoff averages and the breakdown of invariant tori in volume-preserving maps. 
	Phys. D 428 (2021), Paper No. 133048, pp.~20. \url{https://doi.org/10.1016/j.physd.2021.133048}
	
	
	\bibitem[MS25]{MR4845476}
	J. D. Meiss, E. Sander,
	Resonance and weak chaos in quasiperiodically-forced circle maps. Commun. Nonlinear Sci. Numer. Simul. 142 (2025), Paper No. 108562,  pp.~20. \url{https://doi.org/10.1016/j.cnsns.2024.108562}
	
	\bibitem[MRW24]{arXiv:2404.11507}
S. Mondal, J. M. Rosenblatt, M. Wierdl,
	Oscillation of ergodic averages and other stochastic processes. Preprint, 2024.  \url{https://doi.org/10.48550/arXiv.2404.11507}

	
	
	\bibitem[MP21]{MR4201442}
	R. Montalto, M. Procesi,
	Linear {S}chr\"{o}dinger equation with an almost periodic potential.
	SIAM J. Math. Anal. 53 (2021), pp.~386--434. \url{https://doi.org/10.1137/20M1320742}
	
	
	\bibitem[Moo15]{MR3324732}
	C. C. Moore,
	Ergodic theorem, ergodic theory, and statistical mechanics, Proc. Natl. Acad. Sci. USA 112 (2015), pp.~1907--1911.
	\url{https://doi.org/10.1073/pnas.1421798112}
	
	
	

	
\bibitem[Neu32]{vonNeumann}	
	 J. V. Neumann,  Proof of the quasi-ergodic hypothesis.	 Proc Natl Acad Sci USA 18(1) (1932), 70--82.
	
\bibitem[Pod22]{MR4440287}
I. V. Podvigin,  On possible estimates of the rate of pointwise convergence in the Birkhoff ergodic theorem. Translation of Sibirsk. Mat. Zh. 63 (2022), pp.~379--391. Sib. Math. J. 63 (2022), pp.~316--325. \url{https://doi.org/10.1134/s0037446622020094}

\bibitem[Pod24a]{MR4904729}
I. V. Podvigin,  On convergence rates in the Birkhoff ergodic theorem. Sib. Math. J. 65 (2024), pp.~1170--1186. \url{https://doi.org/10.1134/S0037446624050161}


\bibitem[Pod24b]{MR4767798}
I. V.  Podvigin, On the pointwise rate of convergence in the Birkhoff ergodic theorem: recent results. Ergodic theory and dynamical systems, pp.~117--125, De Gruyter Proc. Math., De Gruyter, Berlin, 2024. \url{https://doi.org/10.1515/9783111435503-005}

\bibitem[Pod25]{Podviginnew}
I. V.  Podvigin, On the rate of convergence of $ \mathbb{R}^d $-ergodic averages constructed over a strictly convex set. Preprint, 2025.  \url{https://doi.org/10.48550/arXiv.2506.16740}




\bibitem[PUZ89]{MR1005606}
F. Przytycki, M. Urba\'{n}ski, A. Zdunik, 
Harmonic, Gibbs and Hausdorff measures on repellers for holomorphic maps. I.
Ann. of Math. (2) 130 (1989), pp.~1--40. \url{https://doi.org/10.2307/1971475}	
	
	
		\bibitem[Ryz23]{MR4602432}
			V. V. Ryzhikov, Slow convergences of ergodic averages. (Russian) Mat. Zametki 113 (2023),  pp.~742--746; translation in Math. Notes 113 (2023), pp.~704--707.  \url{https://doi.org/10.4213/mzm13810}

\bibitem[Ryz24]{MR4843319}
V. V. Ryzhikov, Generic correlations and ergodic averages for strongly and mildly mixing automorphisms. (Russian) Mat. Zametki 116 (2024), pp.~438--444; translation in Math. Notes 116 (2024), pp.~521--526. \url{https://doi.org/10.4213/mzm14318}


\bibitem[Ryz25a]{Ryzhikovarxiv}
	V. V. Ryzhikov, Slow convergence of weighted averages for flows and actions of countable amenable groups. Uspekhi Mat. Nauk (Russian Math. Surveys) 80 (2025), pp.~179--180.   \url{https://doi.org/10.4213/rm10251}
	
	
\bibitem[Ryz25b]{Ryzhikovarxiv2}
V. V. Ryzhikov, Slow convergence of Birkhoff ergodic averages.  Preprint, 2025.  \url{https://doi.org/10.48550/arXiv.2507.16740}	
	
\bibitem[Ryz25c]{Ryzhikovarxiv3}
V. V. Ryzhikov,
Slow convergence almost everywhere of ergodic averages. Preprint, 2025. \url{https://doi.org/10.48550/arXiv.2508.00463}

	\bibitem[SM20]{MR4104977}
	E. Sander, J. D. Meiss,
	Birkhoff averages and rotational invariant circles for area-preserving maps.
	Phys. D 411 (2020), 132569, pp.~19. 
	\url{https://doi.org/10.1016/j.physd.2020.132569}
	

	
	
	\bibitem[SM25a]{arXiv:2409.08496}
	E. Sander, J. D. Meiss, Computing Lyapunov exponents using weighted Birkhoff averages. J. Phys. A 58 (2025), Paper No. 355701, pp.~17. \url{https://dx.doi.org/10.1088/1751-8121/adfac1}
	
	\bibitem[SM25b]{MR4853295}
	E. Sander, J. D. Meiss, Proportions of incommensurate, resonant, and chaotic orbits for torus maps. Chaos 35 (2025), Paper No. 013147, pp.~15. \url{https://doi.org/10.1063/5.0226617}
	
	
	
	
	\bibitem[SS03]{MR1970295}
	E. Stein, R. Shakarchi,  
	Fourier analysis.
	An introduction. Princeton Lectures in Analysis, 1. Princeton University Press, Princeton, NJ, 2003.  xvi+311 pp.
	
	\bibitem[Tao08]{MR2408398}
	T. Tao,
	Norm convergence of multiple ergodic averages for commuting transformations. Ergodic Theory Dynam. Systems 28 (2008),  pp.~657--688. \url{https://doi.org/10.1017/S0143385708000011}
	
	

	
		\bibitem[TL24]{MR4768308}
	Z.   Tong, Y. Li,
Exponential convergence of the weighted Birkhoff average. J. Math. Pures Appl. (9) 188 (2024), pp.~470--492. 	\url{https://doi.org/10.1016/j.matpur.2024.06.003}

	
\bibitem[TL25a]{CESAROTL}
	Z.  Tong, Y. Li, A note on exponential convergence of Ces\`{a}ro weighted Birkhoff average and multimodal weighted approach.  Acta Math. Sin. (Engl. Ser.) 41 (2025),  pp.~2301--2323. \url{https://doi.org/10.1007/s10114-025-3365-5}
	
		\bibitem[TL25b]{arXiv:2408.09398}
Z. Tong, Y. Li, Quantitative uniform exponential acceleration of averages along decaying waves. Izv. Ross. Akad. Nauk Ser. Mat. 89 (2025), pp.~131--161; reprinted in Izv. Math. 89 (2025), pp.~1208--1238.  \url{https://doi.org/10.4213/im9666}; \url{https://doi.org/10.4213/im9666e}
	
	
\bibitem[TL25c]{TL25c}
Z. Tong, Y. Li, Weighted Birkhoff averages: Deterministic and probabilistic perspectives. Preprint, 2025. \url{https://doi.org/10.48550/arXiv.2505.03210}



	\bibitem[Yoc80]{MR604672}
	J.-C. Yoccoz,  Sur la disparition de propri\'{e}t\'{e}s de type {D}enjoy-{K}oksma en	dimension {$2$}. C. R. Acad. Sci. Paris S\'{e}r. A--B 291 (1980),  pp.~A655--A658. 
	
	
	\bibitem[Yoc95]{MR1367354}
	J.-C. Yoccoz, Centralisateurs et conjugaison diff\'{e}rentiable des 	diff\'{e}omorphismes du cercle. Petits diviseurs en dimension $1$. Ast\'{e}risque No. 231 (1995), pp.~89--242. 
\end{thebibliography}
\end{document}